\documentclass[11pt]{article}
\usepackage{amsmath, amsfonts, amssymb, mathtools,  amsthm}
\usepackage{xcolor}
\usepackage{float}
\usepackage{enumitem}
\usepackage{cleveref}
\usepackage{euscript}
\usepackage{resmes}

\everymath{\displaystyle}

\usepackage[sans]{dsfont}
\usepackage{mathrsfs}   

\usepackage[T1]{fontenc}
\usepackage[utf8]{inputenc}

\usepackage[osf]{ebgaramond} 

\newtheorem{theorem}{Theorem}[section]

\newtheorem{lemma}[theorem]{Lemma}
\newtheorem{proposition}[theorem]{Proposition}
\newtheorem{corollary}[theorem]{Corollary}
\newtheorem{example}[theorem]{Example}

\newtheorem{claim}[theorem]{Claim}

\theoremstyle{definition}

\numberwithin{equation}{section}
\newtheorem{remark}[theorem]{Remark}

\newcommand{\R}{\mathbb{R}}

\newcommand{\Z}{\mathbb{Z}}

\newcommand{\Sph}{\mathbb{S}}        
\newcommand{\cH}{\mathcal{H}}        
\newcommand{\cF}{\mathcal{F}}
\newcommand{\cP}{\mathcal{P}}        
\newcommand{\Gr}[2]{G(#1,#2)}        
\newcommand{\opV}{\mathds{V}}       

\DeclareMathOperator{\vol}{vol}
\DeclareMathOperator{\MV}{MV}
\DeclareMathOperator{\argmin}{arg\,min}

\DeclarePairedDelimiter{\tnorm}{\lvert\!\lvert\!\lvert}{\rvert\!\rvert\!\rvert}

\allowdisplaybreaks

\begin{document}

\begin{center}
{\LARGE Computing intrinsic volumes of sublevel sets and applications}
\end{center}
\smallskip
\begin{center}
{\Large \textsc{Tr\'i Minh L\^e \& Khai--Hoan Nguyen--Dang}}
\end{center}

\bigskip

\noindent\textbf{Abstract.} 
Intrinsic volumes are fundamental geometric invariants generalizing volume, surface area and mean width for convex bodies.
We establish a unified Laplace--Grassmannian representation for intrinsic and dual volumes of convex polynomial sublevel sets. 
More precisely, let $f$ be a convex $d$--homogeneous polynomial of even degree $d \geq 2$ which is positive except at the origin.
We show that the intrinsic/dual volumes of the sublevel set~$[f \le 1]$ admit Laplace-type integral formulas obtained by averaging the infimal projection and restriction of~$f$ over the Grassmannian. 
This explicit representation yields:
\begin{itemize}
    \item[(i)] L\"owner--John--type existence and uniqueness results, extending beyond the classical volume case; 
    \item[(ii)]  a block decomposition principle describing factorization of intrinsic volumes under direct-sum splitting;
    \item[(iii)] a coordinate-free formulation of Lipschitz-type lattice discrepancy bounds. 
\end{itemize}
The resulting formulas enable analytic treatment for a broad class of geometric quantities, providing direct access to variational and arithmetic applications as well as new structural insights.

\bigskip

\vspace{0.6cm}



\section{Introduction}\label{sec:intro}

Intrinsic volumes $V_j$ (and their dual counterparts $\widetilde V_j$) are central invariants in convex and integral geometry; they interpolate between volume, surface area, mean width and Euler characteristic.
They can be defined via the classical Steiner formula, which expands the volume of the convex body $K + \varepsilon B^n$ as
\[
\vol_n(K + \varepsilon B^n)
  = \sum_{j=0}^n \kappa_{n-j}\, V_j(K)\, \varepsilon^{\,n-j},
\]
where \(\kappa_m\) denotes the volume of the unit \(m\)–ball.
In particular, \(V_{n-1}\) and \( V_1\) are proportional to surface area and mean width, respectively. 
Crucially, they admit elegant integral geometric representations via the Cauchy–Kubota formulas, see e.g. \cite{Schneider,SW_2008}.
In this paper, we study these invariants for \emph{sublevel sets of positively $d$–homogeneous convex polynomials} \(f\in\cP_{n,d}\): 
\[
[f\le \alpha]\ :=\ \{x\in\R^n:\ f(x)\le \alpha\},\quad \ \alpha>0.
\]
Such sets appear naturally as convex models beyond ellipsoids (the case \(d=2\)).
Our goal is to develop \emph{explicit, computable formulas} for
\[
\text{ intrinsic volumes }V_j([f\le\alpha])\quad\text{and dual volumes }\quad \widetilde V_j([f\le\alpha]),
\]
that make transparent the dependence on dimension and degree.
These formulas further facilitate optimization analysis within $\cP_{n,d}$ and provide tools for arithmetic applications such as lattice point enumeration and height computations.

\medskip

Our starting point is a remarkable identity, originally derived via Fourier analysis and observed by Lasserre \cite{La_2015}, which connects the volume of sublevel set of a function $f$ with an integration of $\exp(-f(x))$.
More precisely, we have
\begin{center}
    \textit{If $f : \R^n \to [0, + \infty)$ is a homogeneous polynomial of even degree $d \geq 2$, then it holds
    \begin{equation}\label{lasserre}
        \vol_n([f \leq \alpha]) = \dfrac{\alpha^{n/d}}{\Gamma(1 + n/d)} \int_{\R^n} e^{-f(x)} dx, \quad \text{ for every } \alpha > 0. 
    \end{equation}
    }
\end{center}

\medskip

In the present work, we extend this perspective from volume to all intrinsic and dual intrinsic volumes.
To this end,  we combine this with the projection/section operators defined by, for every $E \in \Gr{j}{n}$,
\[
\Pi_E f(y):=\inf_{z\in E^\perp} f(y+z) \text{ for every  $y \in E$} \quad \text{ and } \quad \mathscr R_E f:=f|_E,
\]
and then average over the Grassmannian \(\Gr{j}{n}\). 
This yields our main representation result, see \Cref{thm.representation},
\begin{equation}\label{LG-repre}
\begin{split}
    V_j([f\le \alpha]) 
    = & ~ \frac{\alpha^{j/d}\beta_{j,n}}{\Gamma(1+j/d)}
    \int_{\Gr{j}{n}}\int_E \exp(-\Pi_E f(y))\,dy\,d\nu_j(E),\\
    \widetilde V_j([f\le \alpha]) 
    = & ~ \ \frac{\alpha^{j/d} \sigma_{j,n}}{\Gamma(1+j/d)}
    \int_{\Gr{j}{n}}\int_E \exp(-\mathscr R_E f(y))\,dy \,d\nu_j(E),
\end{split}
\end{equation}
with the usual normalized constants \(\beta_{j,n}\) and \(\sigma_{j,n}\) from the Cauchy--Kubota formulas.
Here, \(\nu_j\) is the unique Haar measure on \(\Gr{j}{n}\). 
Conceptually, intrinsic and dual intrinsic volumes of \([f\le\alpha]\) are Laplace averages of the infimal projection \(\exp(-\Pi_E f)\) and the restriction \(\exp(-\mathscr R_E f)\) over $\Gr{j}{n}$, respectively.

\medskip

It is worth mentioning that other works have investigated intrinsic volumes of sublevel sets and the role of the projection operator in valuation theory.
For instance, intrinsic volumes for $C^3$--smooth functions on a Riemannian manifold have been obtained through formulas in terms of gradients and Hessians, leading to regularity and continuity results, see~\cite{J_2019}.
In a different direction, the projection operator has been employed in valuation theory for convex supercoercive functions, notably in~\cite{CLM_2022, CLM_2025}, where invariant valuations are characterized via integral representations on subspaces.

\medskip

Although these works share a formal resemblance to our results, they differ fundamentally in scope and methodology: the former relies on smooth differential geometry while the latter concerns classification results in valuation theory (noting that the Laplace–Grassmannian representations are not valuations).
In contrast, our framework combines the Laplace transform with infimal projection and section operators to derive explicit integral formulas for intrinsic and dual volumes, extending beyond both the smooth and axiomatic valuation settings.

\medskip
The Laplace–Grassmannian representations \eqref{LG-repre} serve as the foundation for several new results and applications, which we summarize below.
\medskip

\textbf{Applications}

\medskip

\begin{itemize}
    \item[$\bullet$]
    (\textit{Structural properties})
    Lower-semicontinuity and log-convexity of the functionals $\opV_j : f \mapsto V_j([f \leq 1)$ and $\widetilde\opV_j : f \mapsto \widetilde V_j([f \leq 1])$ follow from Fatou's lemma and H\"older inequality applied on the Grassmannian fibers.
    The key novelty lies in the proof of strict log--convexity of these functionals, see Corollary~\ref{cor:strict-logconvex}.
   The case ${j = n}$ follows directly from the fact that any polynomial equation has finitely many roots (see e.g \cite{La_2015}).
   To treat the case $1 \leq  j \leq n - 1$, our argument proceeds by  a Crofton--type disintegration argument on the incidence manifold
    \begin{align*}
        \mathcal I_j = \big\{ (E, e) \in \Gr{j}{n} \times \mathbb S^{n - 1}~:~ e \in E \big\}.
    \end{align*}
    This allows us to lift the fiberwise equality $\Pi_E f = \Pi_E g$ (or $\mathscr R_E f = \mathscr R_E g$) to the equality of $f$ and $g$ on the unit sphere.
    Therefore, the proof of strict log-convexity of the functionals $\opV_j$ and $\widetilde\opV_j$--which ensures the uniqueness in related optimization problems--emerges from an integral--geometric rather than algebraic argument.

    \item[$\bullet$]
    (\textit{L\"owner--John--type results})
    Fix a compact set $K$ containing the origin in its interior.
    The main purpose of investigating the aforementioned properties is to address the question:
    \begin{center}
        \textit{
        finding a function $f \in \cP_{n, d}$ minimizing $V_j([f \leq 1])$ $($or $\widetilde V_j([f \leq 1]$ $)$ with $K \subset [f \leq 1]$.
        }
    \end{center}
    The case $d = 2$ (quadratic) and $j = n$ (volume) recovers the classical L\"owner--John ellipsoid problem, see e.g \cite{Todd_2016}.
    Thanks to the representation~\eqref{lasserre}, this result extends naturally to arbitrary homogeneous polynomials of even degree beyond the quadratic case, see \cite{La_2015} and further to the class of log--concave functions \cite{GMV_2022}.
    For a convex body \(K\), the quadratic case \(d = 2\) with \(j \in \{1,\dots,n-1\}\) was also considered by Gruber~\cite{G_2008}, who employed a Vorono\u{\i}--type method to establish the existence of a circumscribed ellipsoid minimizing the \(j\)-th intrinsic volumes.
    See also \cite{Schrocker_2008} for another approach to get the uniqueness of such ellipsoids.

    \medskip

    Using the Laplace--Grassmannian representations~\eqref{LG-repre}, we extend this framework to encompass all $ j \in \{1, \cdots, n - 1 \}$ and all even ${d \geq 2}$, 
    thereby unifying the intrinsic and dual volume cases beyond the classical quadratic and volumetric settings.
    We further derive the first--order (KKT) conditions for the associated L\"owner--John--type problems.
    To get the G\^ateaux derivative for KKT conditions, on the section side $\mathscr R_E$, one may differentiate the Laplace--Grassmannian average directly,  while on the projection side $\Pi_E$, one has to use the Danskin--type envelope theorem, see Lemmas~\ref{lem.Vtil-gateaux}--~\ref{lem.V-gateaux}.

    \item
    (\textit{Block factorization})
    A notable feature of the Laplace--Grassmannian representations is their compatibility with orthogonal direct sum.
    As we will show in Proposition~\ref{prop:block-separable}, when the ambient space admits a decomposition
    \[
        \R^n = U_1 \oplus \cdots \oplus U_B, \quad \text{ with $ U_b\leq \R^n$ for every } b \in \{1, \cdots, B \},
    \]
    and $f$ is block--separable, that is, (denoting $P_{U_b}$ the projection onto $U_b$), 
    \begin{align*}
        f(x) = \sum_{b = 1}^B f_b(P_{U_b} x), \quad \text{ where each $f_b$ is a convex positively $d$--homogeneous polynomial on $U_b$},
    \end{align*}
    then the infimal projection and section operators preserve this block separability.
    Consequently, for such  functions $f$, the Laplace--Grassmannian representation factors across active blocks:
    \begin{align*}
    V_j([f \leq \alpha]) = \dfrac{\alpha^{j/d}\beta_{j, n}}{\Gamma(1 + j/d)} \int_{G(j, n)} \prod_{b \in B_E} \int_{E \cap U_b} \exp(- \Pi_{E \cap U_b} f_b(y_b)) \, dy_b \, d\nu_j(E)
\end{align*}
and 
\begin{align*}
    \widetilde V_j([f \leq \alpha]) = \dfrac{\alpha^{j/d}\sigma_{j, n}}{\Gamma(1 + j/d)} \int_{G(j, n)} \prod_{b \in B_E} \int_{E \cap U_b} \exp(- \mathscr R_{E \cap U_b} f_b(y_b)) \, dy_b \, d\nu_j(E),
\end{align*}
where the active index set $B_E$ is defined by $B_E := \{ b \in \{1, \cdots, B \}: U_b \cap E \neq \{ 0 \} \}$.
From a computational standpoint, this property enables one to recover the intrinsic volumes of a high-dimensional separable model from its lower-dimensional factors, providing a concrete route for symbolic and numerical evaluation.

    \item[$\bullet$]
    (\textit{Arithmetic applications})
    A classical problem in analytic number theory is to estimate the number of lattice points inside a large domain:
    \begin{center}
        \textit{Given a compact convex set $K$, one seeks to find asymptotic expansions of the counting function
        \begin{align*}
            N_K(\alpha) := \# \big\{ m \in \mathbb Z^n ~:~ m \in \alpha K    \big\}, \quad \text{ as } \alpha \to \infty.
        \end{align*}
        }
    \end{center}
    This question goes back to Minkowski and has a long history, see e.g \cite{Cassels1959, GruberLekkerkerker1987, Huxley1996, IvicKratzelKuehleitnerNowak2006, Kratzel1989}.

    \medskip

    In recent years, several works have refined the classical discrepancy estimates by connecting lattice counting to convex-analytic invariants such as curvature and intrinsic volumes.
    Classical Lipschitz--type principles, initiated by Davenport~\cite{Davenport1951, Davenport1964Corr}, relate the lattice discrepancy 
    \begin{align*}
        \big\vert \#(K \cap \mathbb Z^n) - \vol_n(K) \big\vert
    \end{align*}
    to the measure of the boundary of $K$.
    Subsequent analytic refinements leading to higher--order estimates have been established via the use of mean--value and second--moment bounds, see \cite{Rogers1955, Rogers1956, Huxley1996}.
    Alternative approaches have also been employed to obtain Lipschitz-type discrepancy bounds, for instance through integral-geometric methods based on Wills functionals \cite{W_1973,H_1975}
    \[
    \mathsf W(K)\ :=\ \sum_{j=0}^{n} V_j(K),
    \]
    and through analytic--arithmetic techniques using Igusa integrals or height estimates \cite{ChambertLoirTschinkel2010, Widmer2010Primitive, Widmer2012Lipschitz}. 
    A parallel thread in high–dimensional geometry and information theory views (conic) intrinsic volumes (and their sums) as complexity parameters controlling phase transitions and average–case behavior of convex signal recovery, see e.g. \cite{ALMT14}. 
    More recently, the Wills functional has also been linked to metric complexity and universal coding rates, see~\cite{Mourtada25}. 

    \medskip

    Coming back to our contributions, we revisit the lattice counting problem for convex polynomial sublevel sets through the lens of the Laplace--Grassmannian representation.
    Although our estimates do not aim at sharp remainder bounds, the resulting asymptotic expansion
    \[
    N_f(\alpha)
    = \opV_n(f)\,\alpha^{n/d}
      + O_{n, f}\!\Big(\alpha^{(n-1)/d}\Big),
      \quad \text{ as } \alpha \to \infty \quad \text{ (see \eqref{operad_Vn} for definition of $\opV_n(f)$)},
    \]
    arises naturally and transparently from the intrinsic-volume structure of the sublevel sets $[f \le \alpha]$.
    This provides a simple route to uniform asymptotic behavior, requiring no delicate analytic estimates beyond the Laplace representation.
    Moreover, the same construction extends seamlessly to other arithmetic counting problems by incorporating the intrinsic and dual intrinsic volume, see Section~5.
    The approach thus offers a conceptually unified and flexible framework connecting convex polynomial and lattice enumeration. 
\end{itemize}

\textbf{Organization.}

\medskip

Section~2 reviews background and fixes notation.
In Section~3, we establish the Laplace--Grassmannian representation. 
Section~4.1 examines some structural properties of the functionals $\opV_j$ and $\widetilde \opV_j$ (see Definitions~\eqref{operad_V}--\eqref{operad_Vtil}): (strict) log--convexity and lower semicontinuity.
Sections~4.2--4.3 study variational theory associated with L\"owner--John--type result and the block–orthogonal factorization, respectively. 
Finally, Section~5 presents arithmetic applications: lattice discrepancy bounds, counts of primitive points and rational subspaces and small--scale theta asymptotics.

\section{Preliminaries}

In this manuscript, we fix two natural numbers:
\begin{align*}
    \quad d \geq 2 \quad \text{ even} \quad \text{ and } \quad n \geq 2.
\end{align*}
We work in the Euclidean space $\R^n$ with the usual inner product $x\cdot y$ and norm $\|x\| :=\sqrt{x\cdot x}$.  
The unit ball and sphere are denoted by
$
B^n := \{x \in \R^n : \|x\| \le 1\} \text{ and } \Sph^{n-1} := \partial B^n
$, respectively.
For $0 \le j \le n$, $\vol_j(\cdot)$ denotes the $j$--dimensional Lebesgue measure (on $j$--flats if $j < n$) and $ \kappa_j := \vol_j(B^j)$.  
If $K \subset \R^n$ and $E \le \R^n$ is a subspace, $K|E$ denotes the orthogonal projection of $K$ onto $E$ and $K \cap E$ its section.  
We also denote $P_E : \R^n \to E$  the orthogonal projection onto a subspace $E$. 

\medskip

For any two sets $A, B \subset \R^n$ and $\lambda \in \R$, their Minkowski sum $A + B$ and the dilation $\lambda A$ are respectively defined by
\[
A + B := \left\{ a + b : a \in A, \, b \in B \right\} \quad \text{ and  } \quad \lambda A := \left\{ \lambda a : a \in A \right\}.
\]

\medskip

We consider the space
\[
    \cH_{n, d} := \left\{ f : \R^n \to \R \, \big\vert \,  f \text{ is a positively $d$--homogeneous polynomial} \right\}
\]
and its convex subsets
\begin{align*}
	\cF_{n, d} := & ~ \{ f \in \cH_{n, d} \, : \,   f(x) > 0 \ \text{for all } x \ne 0 \} \\
	\cP_{n,d} := & ~ \left\{ f \in \cF_{n, d} \, : \, f \text{ is convex } \right\}.
\end{align*}
Here, \textit{positive $d$--homogeneity} refers to
\[
f(\lambda x) = \lambda^d f(x) \quad \text{for all } \lambda > 0,\ x \in \R^n .
\]
Given $f \in \cH_{n,d}$ and $\alpha \in \R$, the sublevel set of $f$ is defined by $[f \le \alpha] := \{ x \in \R^n : f(x) \le \alpha \}.$
Crucially, if $f \in \cP_{n,d}$, then for any $\alpha > 0$, the set $[f \le \alpha]$ is convex, compact and contains the origin in its interior.  
Furthermore, positive $d$--homogeneity yields the scaling on the sublevel set
\[
[f \le \alpha] = \alpha^{1/d} \,[f \le 1] \quad \text{ for every } \alpha > 0.
\]


\textbf{Intrinsic and dual intrinsic volumes.}

\medskip

Let $K_1, \cdots K_m$ be $m$ convex bodies in $\R^n$.
A fundamental result in convex geometry states that the map 
\[
(\lambda_1, \cdots, \lambda_m) \mapsto \vol_n(\lambda_1 K_1 + \cdots \lambda_m K_m)
\]
is a homogeneous polynomial of degree $n$ with respect to Minkowski addition.
More precisely, it holds
\[
\vol_n(\lambda_1 K_1 + \cdots + \lambda_m K_m)
= \sum_{i_1,\dots,i_n=1}^m \MV(K_{i_1},\dots,K_{i_n}) \, \lambda_{i_1}\cdots \lambda_{i_n} \quad \text{for every $\lambda_1, \cdots, \lambda_m > 0$},
\]
where the coefficients $\MV (K_{i_1},\dots,K_{i_n})$ are nonnegative and depend only on the sets $K_{i_j}$ ( $1 \le j \le n$).  
These coefficients are called the \emph{mixed volumes} of $K_{i_1},\dots,K_{i_n}$.  
They are symmetric and multilinear with respect to Minkowski addition and also satisfy
\[
\MV (K,\dots,K) = \vol_n(K).
\]

\medskip



The \emph{intrinsic volumes} $V_j(K)$ for $j \in \{ 0,1,\dots,n \}$ are defined as normalized mixed volumes:
\[
V_j(K) := \frac{1}{\kappa_{n-j}} {n \choose j}
\MV(\underbrace{K,\dots,K}_{j}, \underbrace{B^n,\dots,B^n}_{n-j}).
\]

\medskip

Another crucial characterization is provided by the so--called \emph{Kubota formula}.  
Let $G(j,n)$ denote the Grassmannian of $j$--dimensional linear subspaces of $\R^n$, equipped with the Haar measure $\nu_j$.  
Then
\begin{equation}\label{kubota}
V_j(K) = \binom{n}{j} \frac{\kappa_n}{\kappa_j \kappa_{n-j}}
\int_{\Gr{j}{n}} \vol_j(K | E) \, d\nu_j(E).
\end{equation}

Intrinsic volumes play a central role in convex geometry, valuation theory and integral geometry.
For further details, see the monograph by Schneider~\cite{Schneider}.

\medskip

The notion of \textit{$j$-th dual volumes} (also called the dual $(n - j)$th quermassintegral) can be seen as a dual theory of intrinsic volumes which was first introduced by Lutwak \cite{Lutwak_1975, Lutwak_1979, Lutwak_1988}: for a convex body $K$ containing the origin in its interior, the $j$-th dual intrinsic volume is defined by
\begin{align*}
    \widetilde V_j(K) := \dfrac{1}{n} \int_{\mathbb{S}^{n - 1}} \rho_K(u)^j du,
\end{align*}
where $\rho_K(u) := \max \{ \lambda > 0: \lambda u \in K \}$ is the radial function of $K$.
Analogously to the classical intrinsic volume, we have the Cauchy--Kubota formula for the $j$-th dual volume, expressed in terms of section volumes,
\begin{equation}\label{kubota-dual}
    \widetilde V_j(K) = \dfrac{\kappa_n}{\kappa_{n - j}} \int_{G(j, n)} \vol_j(K \cap E) d\nu_j(E).
\end{equation}

\medskip

From now on, we fix the following constants (which are normalized constants of dual and intrinsic volumes)
 \begin{align*}
        \beta_{j,n} := \binom{n}{j} \,\frac{\kappa_n}{\kappa_j\,\kappa_{n-j}} \quad \text{ and } \quad \sigma_{j, n} := \dfrac{\kappa_n}{\kappa_{n - j}}
\end{align*}

\medskip

\textbf{Lattice.}

\medskip

A \emph{lattice} $\Lambda\subset \R^n$ is a discrete additive subgroup of full rank $n$.
Equivalently, there exists a basis matrix $B=[b_1\,\cdots\,b_n]\in\R^{n\times n}$ with linearly independent columns such that
\[
  \Lambda\ =\ B\,\Z^n\ =\ \left\{ \sum_{i=1}^n m_i b_i:\ m_i\in\Z\right\}.
\]
A measurable set $F\subset\R^n$ is a \emph{fundamental domain} for $\Lambda$ if the translates $\{F+\lambda:\ \lambda\in\Lambda\}$ tile $\R^n$ with pairwise disjoint interiors and
\[
  \R^n\ =\ \bigsqcup_{\lambda\in\Lambda} (F+\lambda).
\]
A canonical choice attached to a basis $B=[b_1\,\cdots\,b_n]$ is the \emph{fundamental parallelepiped}
\[
  \mathcal F(B)\ :=\ \left\{\ \sum_{i=1}^n t_i b_i\ :\ t_i\in[0,1)\ \right\}.
\]
In particular, for the standard lattice $\Lambda=\Z^n$ one may take $F=[0,1)^n$, and $\det\Z^n=1$. 
The covolume (also called the determinant) of $\Lambda$ is
\[
  \det\Lambda\ :=\ \vol_n(\R^n/\Lambda)\ =\ \vol_n\big(\mathcal F(B)\big) = |\det B|
\]
where $\mathcal F(B)$ is any fundamental domain of $\Lambda$. 
This number is independent of the choice of basis $B$.
\section{Exponential representations of intrinsic and dual volumes}

Let $E \in G(j,n)$ be a $j$--dimensional linear subspace of $\R^n$ and let $E^\perp$ denote its orthogonal complement.  
We define two operators acting on $\cP_{n,d}$:
\begin{align*} \label{eq:PiE}
\Pi_E f : E &\longrightarrow \R, 
\qquad \Pi_E f(y) := \inf_{z \in E^\perp} f(y+z), \quad \text{ for every }y \in E,
\\
\mathscr{R}_E f : E &\longrightarrow \R, 
\qquad \mathscr{R}_E f := f|_E .
\end{align*}
Here, $\Pi_E f$ is simply obtained by minimizing over directions orthogonal to $E$  
and $\mathscr{R}_E f$ is the restriction of $f$ to $E$.

\begin{lemma}[Projection/section operators]\label{prop.pro-sec}
For any $0 \leq j \leq n - 1$, let $f\in\cP_{n,d}$ and let $E \in G(j, n)$ be a $j$--dimensional subspace of $\R^n$.
Then, the following assertions hold true:
\begin{enumerate}[label=(\roman*)]
\item $\mathscr R_E f$ is convex, $d$–homogeneous on $E$, and for every $\alpha > 0$, one has
\begin{equation}\label{sec.id}
[f\le \alpha]\ \cap E\ =\ [\mathscr R_E f\le \alpha]\quad\text{(section of sublevel set in $E$).}
\end{equation}

\item $\Pi_E f$ is convex, $d$–homogeneous on $E$, continuous and for every $\alpha > 0$, one has
\begin{equation}\label{proj.id}
[f\le \alpha]\ \big|\ E\ =\ [\Pi_E f\le \alpha]\quad\text{(projection of sublevel set in $E$).}
\end{equation}
\end{enumerate}
\end{lemma}
\begin{proof}
The proof of \textit{(i)} is directly from the definition and therefore we omit it.
\smallskip

\textit{(ii)}
The identity \eqref{proj.id} follows from the definition of $\Pi_E f$.
We now check that $\Pi_E f$ is convex and $d$--homogeneous.
On one hand, since $f$ is convex, we have
\begin{align*}
    \Pi_E f(\lambda y + (1 - \lambda) y') = & ~ \inf_{z \in E^\perp} f(\lambda y + (1 - \lambda) y' + z) \\
    \leq & ~ f(\lambda y + (1 - \lambda) y' + \lambda z + (1  - \lambda) z') \\
    \leq & ~ 
    \lambda f(y + z) + (1 - \lambda) f(y' + z'),
\end{align*}
for every $y, y' \in E$, $z, z' \in E^\perp$ and $\lambda \in (0, 1)$.
Taking the infimum w.r.t $z$ and  $z'$, we infer that $\Pi_E f$ is convex on $E$.
On the other hand, for any fixed $y \in E$ and $t > 0$, using the $d$--homogeneity of $f$ and observing that $E^\perp$ is a linear subspace, we compute
\begin{align*}
    \Pi_E f(t y) = & ~ \inf_{z \in E^\perp} f(t y + z) = \inf_{z \in E^\perp} f(t(y + z)) = t^d \inf_{z \in E^\perp} f(y + z) = t^d \Pi_E f(y).
\end{align*}
Lastly, $\Pi_E f$ is continuous since convex functions are continuous in the interior of their domains, see e.g~\cite[Proposition 1.19]{P_1989}. 
Lemma \ref{prop.pro-sec} is proven.
\end{proof}

\begin{remark}[Projection and section operators for quadratics]\label{rem:PiE-quad}
For $d = 2$, we obtain the explicit expressions for the projection and section operators, together with their integrals over $E$.
Let $Q \in \R^{n\times n}$ be symmetric positive definite and let us consider $f(x) = x^\top Q x$.
Then, $f \in \cP_{n, 2}$ and $[f \le \alpha]=\{x \in \R^n: x^\top Qx\le\alpha\}$ is the centered ellipsoid
\[
\EuScript E_\alpha(Q)\ =\ \sqrt{\alpha}\,Q^{-1/2}\,B^n
\]
with semi–axes $s_i=\sqrt{\alpha/\lambda_i}$, where $0<\lambda_1\le\cdots\le\lambda_n$ are the eigenvalues of $Q$. 
Let $x=y+z$ with $y\in E$ and $z\in E^\perp$.
Then $\Pi_E f(y)=\inf_{z\in E^\perp}(y+z)^\top Q (y+z)$ equals
\[
\Pi_E f(y)\ =\ y^\top\,\big((Q^{-1})|_E\big)^{-1}\,y,
\]
the shorted operator (Schur complement) of $Q$ to $E$. Moreover,
$\mathscr R_E f(y)=y^\top (Q|_E)\,y$. In particular
\begin{align*}
& \int_E \exp(-\Pi_E f(y))\,dy\ =\ \pi^{j/2}\,\sqrt{\det\big((Q^{-1})|_E\big)}\! \\
\text{ and } &
\int_E \exp(-\mathscr R_E f(y)) \,dy\ =\ \frac{\pi^{j/2}}{\sqrt{\det(Q|_E)}}\,.
\end{align*}
\end{remark}

The following lemma provides an integral representation to compute the volume of the sublevel set of a convex, positively $d$-homogeneous polynomial over $E \in G(j,n)$.

\begin{lemma}\label{lem.vol_j}
     Let $E \in G(j, n)$ and let $h: E \to [0, + \infty)$ be a measurable $d$--homogeneous function such that $[h \leq 1]$ is convex and compact. 
     Then, for any $\alpha > 0$, the following identity holds:
     \begin{equation}\label{exp.volume}
         \vol_j([h \leq \alpha]) = \dfrac{\alpha^{j/d}}{\Gamma(1 + j/d)} \int_E \exp(- h(y)) \, dy.
     \end{equation}
\end{lemma}

\begin{remark}
    As we have mentioned in the introduction, the formula~\eqref{exp.volume} has been proved  by Lasserre in \cite[Theorem 2.2]{La_2015} for homogeneous polynomials.
     We provide a shorter proof below using the layer–cake formula for general homogeneous functions, which need not be polynomials.
\end{remark}

\begin{proof}[Proof of Lemma \ref{lem.vol_j}]
Applying Fubini theorem, we have 
\begin{align*}
    \int_E \exp(- h(y)) \, dy = \int_E \int_{h(y)}^\infty e^{-t} \, dt \, dy = \int_0^\infty e^{- t } \vol_j([h \leq t]) \, dt.
\end{align*}
Recall that the $d$--homogeneity of $f$ leads to $[h \leq t] = t^{1/d} [h \leq 1]$.
Hence, using the scaling via dilation of Lebesgue measure, we obtain
    \begin{align*}
        \vol_j([h \leq t]) = t^{j/d}\vol_j([h \leq 1]).
    \end{align*}
Combining the above observations, we arrive at
\begin{align*}
    \int_E \exp(-h(y)) \, dy = \vol_j([h \leq 1]) \int_0^\infty t^{j/d}e^{-t} dt =  \vol_j([h \leq 1]) \Gamma(1 + j/d).
\end{align*}
Finally, we conclude that the identity \eqref{exp.volume} holds true for every $\alpha > 0$.
\end{proof}

Now we can state and prove our main integral representations of intrinsic and dual volumes of the sublevel set of a convex, positively $d$-homogeneous polynomial.
\begin{theorem}[Laplace--Grassmannian representation]\label{thm.representation}
    Let $f \in \cP_{n, d}$ and $\alpha > 0$.
    Then, for any $1 \leq j \leq n - 1$, the following identities hold true:
    \begin{equation}\label{eq:Vj-exp}
        V_j([f\le \alpha]) = \alpha^{j/d}  \,\frac{\beta_{j,n} }{\Gamma(1+j/d)}
\int_{G(j,n)} \ \int_{E} \exp\big(-\Pi_E f(y)\big)\,dy \, d\nu_j(E)
    \end{equation}
    and 
    \begin{equation}\label{eq:dual-exp}
        \widetilde V_j([f\le \alpha])  = \alpha^{j/d} \,\frac{\sigma_{j,n}}{\Gamma(1+j/d)}
        \int_{G(j,n)} \ \int_{E} \exp\big(-\mathscr R_E f(y)\big)\,dy\, d\nu_j(E).
    \end{equation}
    Here $\nu_j$ is the (unique) Haar measure on $G(j, n)$.
\end{theorem}

\begin{proof}
Applying Cauchy--Kubota formula to the intrinsic volume $V_j$, we have that
\begin{align*}
    V_j([f \leq \alpha]) = \beta_{j, n} \int_{G(j, n)} \vol_j([f\le \alpha]\ \big|\ E) \, d\nu_j(E).
\end{align*}
It follows from Lemma \ref{prop.pro-sec} and Lemma \ref{lem.vol_j} that
\begin{align*}
    \vol_j([f\le \alpha]\ \big|\ E) = \vol_j([\Pi_E f \leq \alpha]) =  \dfrac{\alpha^{j/d}}{\Gamma(1 + j/d)} \int_E \exp(- \Pi_E f(y)) \, dy.
\end{align*}
Combining the above observations, we obtain the identity \eqref{eq:Vj-exp}.
Analogously, using the Cauchy--Kubota formula for the $j$-th dual volume \eqref{kubota-dual} together with Lemma \ref{prop.pro-sec} and Lemma \ref{lem.vol_j}, we get the identity \eqref{eq:dual-exp}.
Theorem \ref{thm.representation} is proven.    
\end{proof}

\begin{remark}[Integrability]\label{rem:integrability}
\textit{(i)} For any $f\in \cF_{n,d}$ and any $E\in G(j,n)$, it holds
\[
  \int_E \exp(-\mathscr R_E f(y))\,dy < + \infty\quad \text{ and } \quad 
  \int_E \exp(-\Pi_E f(y))\,dy< +\infty.
\]
Indeed, thanks to Proposition~\ref{prop.cF-coer}, for any fixed $f \in \cF_{n, d}$, there exists $\varpi > 0$ such that $f(x) \geq \varpi \| x \|^d$ for every $x \in \R^n$.
Hence, we get
\begin{align*}
    \int_E \exp(- \mathscr R_E f(y)) \, dy \leq \int_E \exp( - \varpi \| y \|^d) \, dy = \frac{j }{d} \Gamma(j/d) \vol_j(B^j) \varpi^{-j/d} < + \infty
\end{align*}

Furthermore, with a direct computation, we have
\begin{equation*}
    \| y + z \|^2 = \| y \|^2 + 2 \underbrace{\langle y, z \rangle}_{\, = \, 0} + \| z \|^2 \geq \| y \|^2, \quad \text{ for every  } y \in E, \, z \in E^\perp,
\end{equation*}
which leads to $\textstyle\inf_{z \in E^\perp} \| y + z \|^d = \| y \|^d$ for every $y \in E$. 
Remark that this fact may be false if we consider other norms instead of Euclidean norm.
Therefore, we have
\[
    \Pi_E f(y) = \inf_{z \in E^\perp} f(y + z) \geq \varpi \inf_{z \in E^\perp} \| y + z \|^d = \varpi \| y \|^d,
\]
which leads to
\begin{align*}
    \int_E \exp( - \Pi_E f(y)) \, dy \leq \int_E \exp(- \varpi \| y \|^d) \, dy < +\infty.
\end{align*}

\medskip

\textit{(ii)}
As a consequence of \textit{(i)}, since $\nu_j$ is a probability measure on $G(j,n)$, we also have, for every $f \in \cF_{n, d}$,
\begin{align*}
	& \int_{G(j,n)} \ \int_{E} \exp\big(-\Pi_E f(y)\big)\,dy \, d\nu_j(E) < + \infty \\
    \quad \text{ and }   & \int_{G(j,n)} \ \int_{E} \exp\big(-\mathscr R_E f(y)\big)\,dy \, d\nu_j(E) < + \infty.
\end{align*}
\end{remark}


\begin{remark}
Let $K_i=[f_i\le1]$ with $f_i\in\cP_{n,d}$. It is natural to ask whether the \emph{multilinear} mixed volume $\MV(K_1,\dots,K_n)$ admits a comparable \emph{single} exponential integral formula built from $\{f_i\}$. 
In general, it does not.
For a single $f$, $K=[f\le1]$ is a convex sublevel set and $\exp(-f)$ gives the Laplace transform of its (projected/sectional) volumes.
But for distinct $f_i$, the Minkowski sum \(\textstyle\sum_i \lambda_i K_i\) is not of the form $[g\le1]$ for any homogeneous convex polynomial $g$. 
Even in the quadratic case ($d=2$), take centered ellipsoids
\[
\EuScript E(Q):=\{x \in \R^n:\ x^\top Q^{-1}x\le1\}, \quad \text{ whose support function is } h_{\EuScript E(Q)}(u)=\sqrt{u^\top Q\,u}.
\]
Then, one has
\[
h_{\EuScript E(Q_1)+ \EuScript E(Q_2)}(u) = h_{\EuScript E(Q_1)}(u) + h_{\EuScript E(Q_2)}(u)
= \sqrt{u^\top Q_1 u}\,+\,\sqrt{u^\top Q_2 u},
\]
which is the support function of an ellipse only when $Q_1$ and $Q_2$ are homothetic. 
Thus there is no single quadratic $g$ with $[g\le1] = \EuScript E (Q_1) + \EuScript E(Q_2)$.
Without such closure, one cannot hope for a single $'' \exp(-g) ''$-integral that polarizes to the mixed volume coefficients.
\end{remark}

\section{Applications to polynomial optimization}
\subsection{Structural properties of intrinsic/dual volumes of sublevel sets}

\subsubsection*{Log--convexity and lower semicontinuity}

For $1 \leq j \leq n - 1$, let us define respectively the functionals $\opV_j, : \cH_{n, d} \to [0, + \infty]$ and $\widetilde \opV_j : \cH_{n, d} \to [0, + \infty]$ by
\begin{equation}\label{operad_V}
    \qquad\,\,\,\, \opV_j(f) :=  ~
    \begin{dcases}
    \frac{\beta_{j,n} }{\Gamma(1+j/d)}
\int_{G(j,n)} \ \int_{E} \exp\big(-\Pi_E f(y)\big)\,dy \, d\nu_j(E), & f \in \cF_{n, d} \\
    +\infty, & \text{ otherwise }
    \end{dcases},
\end{equation}

\begin{equation}\label{operad_Vtil}
    \text{ and } \quad  \widetilde \opV_j(f) :=  ~ 
    \begin{dcases}
        \frac{\sigma_{j,n}}{\Gamma(1+j/d)}
        \int_{G(j,n)} \ \int_{E} \exp\big(-\mathscr R_E f(y)\big)\,dy \, d\nu_j(E), & f \in \cF_{n, d} \\
        + \infty, & \text{ otherwise }
    \end{dcases}        .
\end{equation}
In the case $j = n$, the above functionals coincide and are simply the volume of the 1--sublevel set of $f$, which is studied by Lasserre~\cite{La_2015}.
More precisely, they are defined by
\begin{equation}\label{operad_Vn}
    \opV_n(f) = \widetilde\opV_n(f) := 
    \begin{dcases}
        \dfrac{1}{\Gamma(1 + n/d)} \int_{\R^n} \exp(- f(x)) dx, & f \in \cF_{n, d}, \\
        + \infty, & \text{ otherwise}
    \end{dcases}.
\end{equation}
In the case $j = 0$, for every $f \in \cP_{n, d}$, we have that $\opV_0(f) = V_0([f \leq 1]) = 1$ (which is in fact the Euler characteristic) and $\widetilde\opV_0(f) = \widetilde V_0([f \leq 1]) = \vol_{n - 1}(\mathbb S^{n - 1})/n$, which are constants.
This case trivializes many subsequent observations and hence, we will not include it in our analysis.

\medskip

Thanks to Remark~\ref{rem:integrability}, we have $\mathrm{dom}~\opV_j = \cF_{n, d}$ and $\mathrm{dom}~\widetilde\opV_j = \cF_{n, d}$.
Moreover, in what follows, we equip the space $\cH_{n, d}$ with the topology of uniform convergence on the unit sphere, which is equivalent to any other norm on $\cH_{n, d}$ since it is a finite-dimensional space.


%

\begin{corollary}[Log–convexity]\label{cor:logconvex}
For any fixed $1 \leq j \leq n$, it holds, for every $ f, g \in \cF_{n, d}$  and  $\lambda \in (0, 1)$,
\[
\begin{aligned}
\opV_j(\lambda f+(1-\lambda)g)
&\ \le\ \opV_j(f)^\lambda\,\opV_j(g)^{1-\lambda},\\[1mm]
\widetilde \opV_j(\lambda f+(1-\lambda)g)
&\ \le\ \widetilde \opV_j(f)^\lambda\,\widetilde \opV_j(g)^{1-\lambda}.
\end{aligned}
\]
\end{corollary}

\begin{proof}
Set $h := (1-\lambda)f+\lambda g$. 
We only prove the inequalities for $1 \leq j \leq n - 1$ and the case $j = n$ follows in the same manner.

\smallskip

\textit{Part 1.} \emph{Log--convex inequality for $\widetilde\opV_j$.} 
For each $E \in \Gr{j}{n}$, since we have $\mathscr R_E h=(1-\lambda)\mathscr R_E f+\lambda\mathscr R_E g$, it holds
\[
e^{-\mathscr R_E h}=(e^{-\mathscr R_E f})^{1-\lambda}(e^{-\mathscr R_E g})^\lambda.
\]
Applying Hölder inequality on $E$ with exponents $p=\frac{1}{1-\lambda}$ and $q=\frac{1}{\lambda}$ gives
\[
\int_E \exp(-\mathscr R_E h(y)) \, dy \le \left(\int_E \exp(-\mathscr R_E f(y)) \, dy \right)^{1-\lambda}
\left(\int_E \exp(-\mathscr R_E g(y)) \, dy \right)^{\lambda}.
\]
Next, applying Hölder inequality again on the Grassmannian $\Gr{j}{n}$ with the same exponents to the functions
\[
    F: E \mapsto \int_E \exp(-\mathscr R_E f(y)) \, dy 
    \quad \text{ and } \quad
    G: E \mapsto \int_E \exp(-\mathscr R_E g(y)) \, dy,
\]
we obtain that
\[
\int_{G(j,n)}\!\int_E \exp(-\mathscr R_E h(y))\, dy \, d\nu_j(E)\ \le\
\Big(\int_{G(j,n)} F(E) \, d\nu_j(E)\Big)^{1-\lambda}
\Big(\int_{G(j,n)} G(E) \, d\nu_j(E)\Big)^{\lambda}.
\]
Then, the representation formula in Theorem~\ref{thm.representation} yields the desired inequality for $\widetilde \opV_j$.

\medskip

\textit{Part 2.} \emph{Log--convex inequality for $\opV_j$.} For each $E \in G(j, n)$, note that the projection operator $\Pi_E$ is concave in terms of $f$:
\[
\Pi_E \, h(y)\ \ge\ (1-\lambda)\Pi_E \, f(y) + \lambda\Pi_E \,  g(y), \quad \text{ for every } y \in E.
\]
Since $r\mapsto e^{-r}$ is decreasing in $[0, + \infty)$, we infer that
\[
e^{-\Pi_E h(y)}\ \le\ \big( e^{-\Pi_E f(y)} \big)^{1-\lambda} \big(e^{-\Pi_E g(y)} \big)^\lambda, \quad \text{ for every } y \in E.
\]
The remaining proof follows similarly as in Part 1.
\end{proof}

\begin{corollary}[Lower semicontinuity]\label{cor.lsc}
    For any fixed $1 \leq j \leq n$, the functionals $\opV_j$ and $\widetilde \opV_j$ are lower semicontinuous on their domains.
\end{corollary}

\begin{proof}
The case $j = n$ has been proved in \cite{La_2015}.
Let us fix $1 \leq j \leq n - 1$.
By analogy, it suffices to verify the lower semicontinuity for the functional $\opV_j$.
Let $\{ f_n \} \subset \cF_{n, d}$ and $f \in \cF_{n, d}$ be such that $f_n \to f$ as $n \to \infty$.
By homogeneity, we have $f_n(x) \to f(x)$ as $n \to \infty$ for every $x \in \R^n$.
To proceed, we now show that
\begin{equation}\label{limsup.PiE}
    \Pi_E f(y) \geq \limsup_{n \to \infty} \Pi_E f_n(y),  \quad \text{ for every } E \in \Gr{j}{n} \text{ and } y \in E.
\end{equation}
Fix $y \in E$.
One can then find a subsequence $\{ n_k \}$ such that  $\textstyle\limsup_{n \to \infty} \Pi_E f_n(y) = \textstyle\lim_{k \to \infty} \Pi_E f_{n_k}(y)$.
For any $\varepsilon > 0$, by definition of the operator $\Pi_E$, there exists $z_\varepsilon \in E^\perp$ satisfying
\[
    \Pi_E f(y) + \varepsilon > f(y + z_\varepsilon) = \lim_{k \to \infty} \underbrace{f_{n_k}(y + z_\varepsilon)}_{\, \geq \, \Pi_E f_{n_k}(y)} \geq  \limsup_{n \to \infty} \Pi_E f_n(y).
\]
Since $\varepsilon > 0$ is arbitrary, the limsup inequality \eqref{limsup.PiE} follows.
Combining the representation in Theorem~\ref{thm.representation} and Fatou's lemma, we obtain the lower semicontinuity of $\opV_j$:
\begin{align*}
    \liminf_{n \to \infty} \opV_j(f_n)  = & ~ \dfrac{\beta_{j, n}}{\Gamma(1 + j/d)} \liminf_{n \to \infty} \int_{\Gr{j}{n}}\int_E \exp(- \Pi_E f_n(y)) \, dy \, d\nu_j(E) \\
    \geq & ~ \dfrac{\beta_{j, n}}{\Gamma(1 + j/d)}  \int_{\Gr{j}{n}}\int_E \liminf_{n \to \infty} \exp(- \Pi_E f_n(y)) \, dy \, d\nu_j(E) \\
    \geq & ~ 
    \dfrac{\beta_{j, n}}{\Gamma(1 + j/d)}  \int_{\Gr{j}{n}}\int_E \exp(- \Pi_E f(y)) \, dy \, d\nu_j(E) = \opV_j(f),
\end{align*}
where we have used the limsup inequality~\eqref{limsup.PiE} in the last estimate. 
Corollary~\ref{cor.lsc} is proven.
\end{proof}

\subsubsection*{Strict log-convexity through a disintegration argument}
In what follows, we are confirming the strict log--convexity of the functionals $\opV_j$ and $\widetilde \opV_j$, which is sharper than Corollary \ref{cor:logconvex}.
To do so, we use a \textit{disintegration argument} over the double fibration:
\[
G(j,n)\ \xleftarrow{\ \pi_1\ }\ \mathcal I_j\ \xrightarrow{\ \pi_2\ }\ \Sph^{n-1},
\]
where $\mathcal I_j$ is the incidence manifold, see, e.g., \cite[Chapters ~7, 13]{SW_2008} for the integral-geometric framework and \cite[Chapter~10]{Bogachev_2007} for the general disintegration theorem.

\medskip

We now describe the construction precisely.
Recall that $\nu_j$ is the Haar measure on $G(j, n)$ and let $\sigma$ be the surface area on $\Sph^{n - 1}$.
Let
\[
  \mathcal I_j\ :=\ \{(E, e)\in G(j,n)\times \Sph^{n-1} ~:~ e\in E\}
\]
be the incidence manifold, with the canonical projections
\[
  \pi_1:\mathcal I_j\to G(j,n),\quad \pi_1(E, e) = E
  \quad \text{ and } \quad
  \pi_2:\mathcal I_j\to \Sph^{n-1},\quad \pi_2(E, e) = e.
\]
We view $\mathcal I_j$ as a measurable subset of the product space $G(j,n)\times \Sph^{n-1}$ and
use the product $\sigma$--algebra throughout.
With these constructions, one obtains a Crofton--type disintegration formula as follows.
Let us begin with the $1$-dimensional case.

\begin{lemma}[Incidence disintegration: the case \(j=1\)]\label{lem:incidence-1}
For each $E \in G(1, n)$, denote $\sigma_E$ the counting measure on the set $E \cap \mathbb S^{n - 1} = \{ \pm e_E \}$.
Set
\[
    c_{1, n} := \dfrac{2}{n \kappa_n}.
\]
Then, for every nonnegative Borel function
\(\varphi:G(1,n)\times\mathbb S^{n-1}\to[0,\infty)\), it holds
\begin{equation}\label{eq:j1-functional}
\int_{G(1,n)} \;\int_{\mathbb S^{n-1}\cap E} \varphi(E, e)\,d\sigma_E(e)\,d\nu_1(E)
\;=\;
c_{1,n}\,\int_{\mathbb S^{n-1}} \varphi\bigl(\operatorname{span}(e),e\bigr)\,d\sigma(e).
\end{equation}
In particular, for every Borel set \(A\subset\mathbb S^{n-1}\), it holds
\begin{equation}\label{eq:j1-incidence}
\int_{G(1,n)} \sigma_E\bigl(A\cap E \bigr)\,d\nu_1(E)
= c_{1,n}\,\sigma(A).
\end{equation}
\end{lemma}

\begin{proof}
Recall that $\mathcal I_1 = \left\{ (E,u)\in G(1,n)\times \mathbb S^{n-1} ~:~ u\in E \right\}$ is the incidence set.
Let $\mathcal H^0$ be the Hausdorff measure of dimension zero.
Define a finite Borel measure \(\rho\) on \(\mathcal I_1 \) by
\[
\rho(B) \,:=\, \int_{G(1,n)} \mathcal H^{0}\!\bigl(\{u\in \mathbb S^{n-1}\cap E:(E,u)\in B\}\bigr)\,d\nu_1(E), \quad \text{ for every Borel set } B \subset \mathcal I_1
\]
By construction, \(\rho\) is \(O(n)\)-invariant for the diagonal action on \(\mathcal I_1\).
More precisely, for every Borel set $B$ and every ${R \in O(n)}$, it holds $\rho(RB) = \rho(B)$, where $R$ acts on the pair $G(1, n) \times \mathbb S^{n - 1}$ by $R(E, u) = (RE, Ru)$, i.e. rotates the subspace 
$E$ and the vector $u$ simultaneously.
Consider also the \(O(n)\)-equivariant map
\begin{align*}
\iota:\mathbb S^{n-1} & \to\mathcal I_1, \\
 e & \mapsto \iota(e) = (\operatorname{span}(e),e).
\end{align*}
Define the pushforward measure
\[
\tilde\rho(B)\;:=\;\sigma\bigl(\iota^{-1}(B)\bigr)\;=\;\int_{\mathbb S^{n-1}}\mathbf 1_{B}\bigl(\operatorname{span}(e),e\bigr)\,d\sigma(e),
\quad \text{ for every Borel set } B \subset \mathcal I_1.
\]
Consequently, we infer that \(\tilde\rho\) is also a finite \(O(n)\)-invariant measure on \(\mathcal I_1 \).

\medskip

Since \(O(n)\) acts transitively on \( \mathcal I_1\), the space of invariant finite measures on \( \mathcal I_1 \) is one–dimensional.
Hence there exists \(c>0\) with
\(\rho=c\,\tilde\rho\).
Evaluating on \(\mathcal I_1 \) gives
\[
\rho(\mathcal I_1) \,=\, \int_{G(1,n)} \mathcal H^{0}\!\bigl(\mathbb S^{n-1}\cap E\bigr)\,d\nu_1(E)
\;=\; \int_{G(1,n)} 2\,d\nu_1(E) \;=\; 2,
\]
whereas \(\tilde\rho(\mathcal I_1 ) = \sigma(\mathbb S^{n-1})=n\,\kappa_n\). 
This verifies the choice of normalized constant  \(c = c_{1, n} =2/(n\,\kappa_n) \).

\medskip

Under the identity $\rho = c_{1, n} \widetilde \rho$, for any nonnegative Borel \(\varphi\) we obtain by definition of \(\rho\) and \(\tilde\rho\)
\[
\int_{\mathcal I_1} \varphi(E,u)\,d\rho(E,u)
\, = \,
c_{1,n}\,\int_{\mathcal I_1}\varphi(E,u)\,d\tilde\rho(E,u)
\, = \,
c_{1,n}\,\int_{\mathbb S^{n-1}} \varphi\bigl(\operatorname{span}(e),e\bigr)\,d\sigma(e),
\]
which is precisely \eqref{eq:j1-functional}. 
Notice that $\sigma_E$ is simply the restriction of $\mathcal H^0$, i.e, $\sigma_E = \mathcal H^0 \resmes (\mathbb S^{n - 1} \cap E)$.
Finally, for any Borel set $A \subset \mathbb S^{n - 1}$, taking \(\varphi(E, e)=\mathbf 1_A( e )\) yields \eqref{eq:j1-incidence}. This completes the proof.
\end{proof}

Generalizing above strategy, we obtain a disintegration argument for all $2 \leq j \leq n - 1$.

\begin{lemma}[Incidence disintegration: the case $2 \leq j \leq  n  - 1$]\label{lem:incidence}
Let $2 \leq j \leq  n - 1$.
For $E\in G(j,n)$ write $\sigma_E$ for the $(j-1)$–dimensional surface measure on the subsphere $\Sph^{n-1}\cap E$. 
Then, there exist a constant $c_{j,n}>0$ and, for $\sigma$–a.e.\ $e\in \Sph^{n-1}$, a Borel probability measure $\mu_e$ supported on the fiber
\[
\mathcal F_e\ :=\ \{E\in G(j,n):\ e\in E\}\ \cong\ G(j-1,n-1),
\]
such that for every nonnegative Borel $\varphi:\mathcal I_j\to\R$,
\begin{equation}\label{eq:incidence-disintegration}
  \int_{G(j,n)}\ \int_{\Sph^{n-1}\cap E}\!\varphi(E,e)\,d\sigma_E(e)\,d\nu_j(E)
  \ =\
  c_{j,n}\ \int_{\Sph^{n-1}}\ \int_{\mathcal F_e}\! \varphi(E,e)\,d\mu_e(E)\,d\sigma(e).
\end{equation}
Moreover, it holds
\begin{equation}\label{eq:cjn}
  c_{j,n}\ =\ \frac{j\,\kappa_j}{n\,\kappa_n}\,,
\end{equation}
and 
$\mu_e$ can be chosen as the unique $SO(n-1)$–invariant probability on $G(j-1,n-1)$.
Consequently, for every Borel set $A\subset \Sph^{n-1}$,
\begin{equation}\label{eq:incidence-indicator}
 \int_{G(j,n)} \sigma_E(A\cap E)\,d\nu_j(E)\ =\ c_{j,n}\,\sigma(A).
\end{equation}
\end{lemma}

\begin{proof}
Since $G(j,n)$ and $\Sph^{n-1}$ are compact metric spaces and $\mathcal I_j$ is a closed subset of the product, it follows that $\mathcal I_j$ is also a compact metric space. 
Thanks to Riesz–Markov–Kakutani representation theorem, there exists a finite Borel measure $\mathsf M$ on $\mathcal I_j$ defined by
\begin{equation}\label{eq:Ij-measure}
  \int_{\mathcal I_j}\!\varphi\,d\mathsf M
  \ :=\ \int_{G(j,n)}\ \int_{\Sph^{n-1}\cap E} \varphi(E,e)\,d\sigma_E(e)\,d\nu_j(E),
  \quad \text{ for every } \varphi\in C_c(\mathcal I_j).
\end{equation}
Moreover, we also have $\mathsf M(\mathcal I_j) = j \kappa_j$, corresponding to the surface area of $\mathbb S^{j - 1}$ in $E$, and by construction, the measure $\mathsf M$ is $SO(n)$--invariant. 
The projection $\pi_2 : \mathcal I_j\to \Sph^{n-1}$ is continuous; hence, its pushforward measure $\pi_{2\#}\mathsf M$ is well--defined and is finite since $\pi_{2\#} \mathsf M(\mathbb S^{n - 1}) = \mathsf M(\mathcal I_j) < + \infty$.
Due to the $SO(n)$–invariance of $\mathsf M$, $\pi_{2\#}\mathsf M$ is also $SO(n)$--invariant.
Therefore, there exists a constant $c_{j, n} > 0$ such that
\[
  \pi_{2\#}\mathsf M \ =\ c_{j,n}\,\sigma.
\]
Evaluating both sides on $\mathbb S^{n - 1}$ gives
\[
    c_{j, n} = \dfrac{j \kappa_j}{n \kappa_n}.
\]

\medskip

Normalize $\mathsf M$ to a probability $\widehat{\mathsf M}:=\mathsf M/\mathsf M(\mathcal I_j)$ and likewise normalize $\sigma$ to a probability $\widehat\sigma:=\sigma/\sigma(\Sph^{n-1})$. 
Since $\mathcal I_j$ is a compact metric space (and hence complete and separable), $\widehat{\mathsf M}$ is a Radon probability measure on $\mathcal I_j$ and therefore perfect (see \cite[Theorems 7.1.7 and 7.5.10]{Bogachev_2007}).
According to \cite[Theorem 7.5.6--(iv)]{Bogachev_2007}, every perfect measure on a countably separated $\sigma$–algebra possesses a compact approximating class (see \cite[Definition 1.4.6]{Bogachev_2007}); hence, $\widehat{\mathsf M}$ admits one. 
Since $\widehat{\mathsf M}$ admits a compact approximating class, we are able to apply \cite[Corollary 10.6.7]{Bogachev_2007} on the existence of disintegration to the sub--$\sigma$--algebra $\pi_2^{-1}\mathcal B(\Sph^{n-1})$.
Therefore, there exists a family of conditional probability measures $\{\mu_e\}_{e\in \Sph^{n-1}}$ on the fibers such that for all bounded Borel $\varphi$,
\[
  \int_{\mathcal I_j}\! \varphi\,d\widehat{\mathsf M}
  \ =\ \int_{\Sph^{n-1}} \left(\int_{\pi_2^{-1}(\{e\})}\! \varphi\,d\mu_e\right) d\widehat\sigma(e).
\]
Returning to the original measures and recalling that $\pi_{2\#}\mathsf M=c_{j,n}\sigma$, we obtain the identity~\eqref{eq:incidence-disintegration}. 

\medskip

It remains to check the choice of $\mu_e$.
Since $\pi_2$ is continuous, its graph is Borel in $\mathcal I_j\times \Sph^{n-1}$.
Therefore, the conditional measures $\mu_e$ are concentrated on the fiber $\pi_2^{-1}(\{e\})=\{(E,e):e\in E\}$ for $\sigma$–a.e.\ $e \in \mathbb S^{n - 1}$, see e.g. \cite[Corollary 10.5.7]{Bogachev_2007} (regular conditional probabilities are supported on the fibers). 
Identifying the fiber with $G(j-1,n-1)$ via $E\mapsto E\cap e^\perp$ (with inverse $F \mapsto \mathrm{span}(e) \oplus F$) shows that $\mu_e$ is a probability on $G(j-1,n-1)$. 
Let $R\in SO(n)$ fix $e$, i.e., $Re=e$. Then $R$ maps the fiber $\mathcal F_e$ to itself. The $SO(n)$–invariance of $\mathsf M$ implies (by uniqueness in the disintegration) that $\mu_e$ is invariant under the action of such $R$, i.e., under the full stabilizer $SO(n-1)$ of $e$. 
Hence, $\mu_e$ is the unique $SO(n-1)$–invariant (Haar) probability on the homogeneous space $G(j-1,n-1)$. 
Finally, \eqref{eq:incidence-indicator} follows directly from \eqref{eq:incidence-disintegration} applied to 
\(\varphi = \mathbf{1}_{\{(E,e):\, e \in A\}}\), 
since \(\varphi(E,e) = \mathbf{1}_A(e)\) is constant on each fiber.
Lemma~\ref{lem:incidence} is proven.
\end{proof}

As a direct consequence of the identity \eqref{eq:incidence-indicator}, we get the following corollary.

\begin{corollary}[Fiberwise nullity]\label{cor:incidence-positivity}
Let $1 \leq j \leq n - 1$ and let $\sigma_E$ be defined as in Lemmas \ref{lem:incidence-1}--\ref{lem:incidence}.
For any Borel set $A \subset \Sph^{n-1}$, the following assertions hold:
\begin{itemize}
    \item[(i)] $\sigma(A) = 0$ if and only if 
    $\sigma_E(A \cap E) = 0 \ \text{for}\ \nu_j\text{--a.e. } E \in G(j,n)$;
    
    \item[(ii)] $\nu_j(\{E \in \Gr{j}{n} : A \cap E \neq \varnothing\}) = 0$ implies that $\sigma(A)=0$.
\end{itemize}
\end{corollary}

\begin{proof}
Recall from Lemmas~\ref{lem:incidence-1}--\ref{lem:incidence} that 
\[
\int_{G(j,n)} \sigma_E(A \cap E)\,d\nu_j(E)\ =\ c_{j,n}\,\sigma(A) \quad 
\text{ for some constant $c_{j, n} > 0$.}
\]

\textit{(i)} 
If $\sigma(A)=0$, then as $\sigma_E(A \cap E) \geq 0$, the above identity yields
$\sigma_E(A \cap E)=0$ for $\nu_j$–a.e\ $E$.
Conversely, if $\sigma_E(A \cap E)=0$ for $\nu_j$–a.e\ $E$, the left-hand
side is zero and so $c_{j,n}\,\sigma(A)=0$.
Therefore $\sigma(A)=0$.

\medskip

\textit{(ii)} If one has $\nu_j(\{E \in \Gr{j}{n}: A\cap E\neq\varnothing\})=0$, then $A\cap E=\varnothing$ for $\nu_j$–a.e.\ $E$.
Hence, we have $\sigma_E(A \cap E) =\sigma_E(\varnothing) = 0$ for $\nu_j$–a.e.\ $E$. Apply (i) to conclude $\sigma(A)=0$.
\end{proof}

\begin{remark}\normalfont
Notice that $\sigma(A)=0$ may not imply $\nu_j (\{E \in \Gr{j}{n}: (A\cap E)\neq\varnothing\})=0$.
Albeit simple, let us consider the following example.
Consider $n = 3$ and fix $E_0 \in G(2, 3)$.
Set $A = E_0 \cap \Sph^{2}$, which is a great circle and therefore $\sigma(A) = 0$.
However, for every two--dimensional plane $E \neq E_0$, $E \cap E_0$ is the whole line and so $A \cap E = \Sph^2 \cap E \cap E_0$ is exactly two antipodal points. 
Therefore, $A \cap E \neq \varnothing$ for every $E \neq E_0$. 
This implies that $\nu_2(\{ E \in G(2, 3) : A \cap E \neq \varnothing \}) = 1$.
\end{remark}

\begin{corollary}[Strict log-convexity]\label{cor:strict-logconvex}
For $1\le j \leq n$, the maps $\widetilde\opV_j$ and $\opV_j$ are strictly log--convex on $\cF_{n,d}$. 
\end{corollary}

\begin{proof}
The case $j = n$ can be handled as in \cite[Theorem 2.4]{La_2015} due to the fact that the solution set of any polynomial equation has finitely many points.
In what follows, we focus on the case $1 \leq j \leq n - 1$.
We split the proof into two parts.

\smallskip

\textbf{Part 1:} \textit{$\widetilde\opV_j$ is strictly log--convex.}

\smallskip

A careful inspection of the proof of Corollary~\ref{cor:logconvex} shows, equality in the log–convex inequality occurs precisely when equality holds in Hölder’s inequality.
Therefore $\widetilde \opV_j(\lambda f+(1-\lambda)g)
= \widetilde \opV_j(f)^\lambda\,\widetilde \opV_j(g)^{1-\lambda}$ holds if and only if for $\nu_j$--a.e $E \in \Gr{j}{n}$, there exists $c_E \geq 0$ satisfying $\exp(-\mathscr R_E f) = c_E\, \exp(-\mathscr R_E g)$ a.e.\ on $E$.
Notice that $\mathscr R_E \, f$ and $\mathscr R_E \, g$ are continuous and so the previous equality holds true for every  $x \in E$. 
Evaluating at the origin yields $c_E =  1$.
Therefore, it must hold $\mathscr R_E f = \mathscr R_E g$ for $\nu_j$--a.e $E$.
Applying Corollary~\ref{cor:incidence-positivity}--(ii) to the set 
\[
A = \{ e \in \Sph^{n - 1}: (f - g)(e) \neq 0 \},
\]
we immediately infer that $f = g$ $\sigma$--a.e on $\Sph^{n - 1}$. 
Lastly, due to the homogeneity and continuity of $f$ and $g$, we obtain the equality $f = g$ in $\R^n$.
Therefore, the log--convex inequality of $\widetilde \opV_j$ is strict whenever $f \neq g$. 

\medskip

\textbf{Part 2:} \textit{$\opV_j$ is strict log--convex.}

\smallskip

Assume that  $\opV_j(\lambda f+(1-\lambda)g)
= \opV_j(f)^\lambda\, \opV_j(g)^{1-\lambda}$.
Similar arguments to those in Part 1 imply that $\Pi_E f = \Pi_E g$ on $E$ and hence $[ \Pi_E \, f \leq 1 ] = [\Pi_E \,  g \leq 1]$ for $\nu_j$--a.e $E$.
Thanks to Lemma \ref{prop.pro-sec}--\textit{(ii)}, we know that 
\[
[f \leq 1] \vert E = [g \leq 1] \vert E, \quad \text{ for $\nu_j$--a.e $E$}.
\]

\begin{claim}\label{claim.supp-proj}
Let $K, L$ be two convex bodies such that $K \vert E = L \vert E$ for $\nu_j$--a.e $E$.
Then, it holds 
\[
	h_K(e) = h_L(e), \quad \text{ for $\sigma$--a.e $e \in \mathbb S^{n - 1}$},
\]
where $h_M$ denotes the support function of a convex body $M$.
\end{claim}

\textit{Proof of Claim \ref{claim.supp-proj}}

We first observe that for any $E \in \Gr{j}{n}$ and $y \in E$, it holds
\begin{align*}
	h_{K \vert E}(y) = \sup_{y' \in K \vert E} \langle y, y' \rangle = \sup_{y' \in K} \underbrace{\langle y, P_E y' \rangle}_{\, = \, \langle y, y' \rangle} = h_K(y), 
\end{align*}
where $P_E y'$ is the projection of $y'$ onto $E$.
Denote $A = \{ e \in \Sph^{n - 1}: h_K(e) \neq h_L(e) \}$.
Then, the above observation leads to
\begin{align*}
	\EuScript A := \{ E \in \Gr{j}{n}: A \cap  E \neq \varnothing \} = & ~ 
	\{ E \in \Gr{j}{n}: \exists \, e \in \Sph^{n - 1} \cap E, \, h_K(e) \neq h_L(e) \}
	\\
	= & ~ 
	\{ E \in \Gr{j}{n}:  \exists \, e \in \Sph^{n - 1} \cap E, \, h_{K \vert E}(e) \neq h_{L \vert E}(e) \} \\
	= & ~
	\{ E \in \Gr{j}{n}:  K \vert E \neq L \vert E \}.
\end{align*}
Using the assumption $K \vert E = L \vert E$ for $\nu_j$--a.e $E$, we infer that $\nu_j(\EuScript A) = 0$.
Therefore, applying Corollary~\ref{cor:incidence-positivity}--\textit{(ii)}, we deduce that $\sigma(A) = 0$, which completes the proof of Claim~\ref{claim.supp-proj}. \hfill $\lozenge$

\medskip

Coming back to the proof of Step 2, applying Claim~\ref{claim.supp-proj} to the case $K = [f \leq 1]$ and $L = [g \leq 1]$, we infer that $h_{[f \leq 1]} = h_{[g \leq 1]}$ for $\sigma$--a.e on $\Sph^{n - 1}$.
Notice that the support function of a convex body is continuous.
Thus, $h_{[f \leq 1]} = h_{[g \leq 1]}$ in $\Sph^{n - 1}$ and so we have $[f \leq 1] = [g \leq 1]$.
It follows that $f = g$ in $\R^n$. 
In conclusion, we have proved that the log--convex inequality of $\opV_j$ is strict whenever $f \neq g$.
Corollary~\ref{cor:strict-logconvex} is proven.
\end{proof}

\subsection{Analogues of L\"owner--John ellipsoids}
In this section, we are interested in studying homogeneous polynomials which minimize the sublevel set containing a given compact set. 
More precisely, for any compact set $K$ containing the origin in its interior, we consider the following problems
\begin{equation}\tag{$P_0$}\label{P-0}
    \text{minimize} \quad V_j([f \leq 1]) \quad \text{ such that } \quad f \in \cP_{n, d} \text{ and }K \subset [f \leq 1],
\end{equation}
and
\begin{equation}\tag{$\widetilde{P}_0$}\label{Ptil-0}
    \text{minimize} \quad \widetilde V_j([f \leq 1]) \quad \text{ such that } \quad f \in \cP_{n, d} \text{ and }K \subset [f \leq 1].
\end{equation}
The above problems extend the classical L\"owner--John ellipsoid problem in two directions.
On the one hand, instead of considering volume, we minimize more general geometric quantities such as intrinsic and dual volumes. 
On the other hand, the feasible sets are no longer restricted to ellipsoids, but to sublevel sets of polynomials of fixed even degree $d \ge 2$.
In this way, the formulation generalizes the Löwner–John framework, albeit within the limited class of sublevel sets of polynomials due to our approach.

\begin{proposition}\label{prop.exi-uni-JOHN}
    Let $K$ be a compact set containing the origin in its interior and let $1 \leq j \leq n$.
   The problems \eqref{P-0} (resp. \eqref{Ptil-0}) admits a unique solution $f^\star \in \cP_{n, d}$ (resp. $\widetilde{f}^\star \in \cP_{n, d}$).
\end{proposition}

\begin{proof}
    Thanks to Theorem~\ref{thm.representation}, observe first that the problems \eqref{P-0} and \eqref{Ptil-0} are respectively equivalent to
    \begin{equation}\label{P-0-modi}
        \text{minimize} \quad \opV_j(f) \quad \text{ such that } \quad f \in \cP_{n, d} \text{ and }K \subset [f \leq 1],
    \end{equation}
    and
    \begin{equation}
        \text{minimize} \quad 
        \widetilde\opV_j(f) \quad \text{ such that } \quad f \in \cP_{n, d} \text{ and }K \subset [f \leq 1].
    \end{equation}
    We shall show the existence of $f^\star$ and a similar proof can be applied to get $\widetilde f^\star$. 
    We proceed by using direct method.
    Let $\{ f_k \}$ be a minimizing sequence of \eqref{P-0-modi}.
    Observe that there exists $r > 0$ such that $rB^n \subset K$ since $K$ contains the origin in its interior.
    Thanks to the $d$--homogeneity of $f$, we have
        \[
            f_k(x) = \| x \|^d f_k\left( \dfrac{x}{\| x \|} \right) = r^d f_k \left( \dfrac{x}{\| x \|} \right), \text{ for every } x \in r \partial B^n,
        \]
    and due to the constraint $K \subset [f_k \leq 1]$, we infer that
        \[
            f_k(e) \leq \frac{1}{r^d} \qquad \text{ for every } e \in \mathbb S^{n - 1} \text{ and } k \in \mathbb N.
        \]
    Note that $\cH_{n, d}$ is a finite--dimensional linear space. 
    The above observation implies that $\{ f_k \} $ is a bounded sequence in $( \cH_{n, d}, \| \cdot \|_{L^\infty(\mathbb S^{n - 1})})$.
    Therefore, there exist a subsequence $\{ f_{k_\ell} \}$ and $f^\star \in \cH_{n, d}$ such that $f_{k_\ell} \to f^\star$ as $\ell \to \infty$ uniformly in $\mathbb S^{n - 1}$.
    It is straightforward to check that $K \subset [f^\star \leq 1]$ and $f^\star \in \cP_{n, d}$.
    Applying Corollary~\ref{cor.lsc}, we deduce
        \begin{align*}
            \inf_{\eqref{P-0}} \opV_j(f) \leq \opV_j(f^\star) \leq \liminf_{\ell \to \infty} \opV_j(f_{k_\ell}) = \inf_{\eqref{P-0}} \opV_j(f),
        \end{align*}
    which implies that $f^\star$ is a minimizer of \eqref{P-0}.
    The uniqueness of $f^\star$ follows directly from the strict log--convexity of $\opV_j$ in Corollary \ref{cor:strict-logconvex}, which completes the proof.
\end{proof}

\begin{lemma}[G\^ateaux derivative of $\widetilde\opV_j$]\label{lem.Vtil-gateaux}
    Let $f \in \cP_{n, d}$ and let $1 \leq j \leq n - 1$.
    Then, $\widetilde\opV_j$ is G\^ateaux differentiable and moreover for any $\phi \in \cH_{n, d}$ one has
    \begin{align*}
       \mathrm d \widetilde\opV_j(f; \phi) = - \dfrac{\sigma_{j, n}}{\Gamma(1 + j/d)} \int_{G(j, n)}\int_E (\mathscr{R}_E \phi)(y) \exp(- \mathscr R_E f(y)) \, dy \, d\nu_j(E).
    \end{align*}
\end{lemma}

\begin{proof}
The case $\phi \equiv 0$ is vacuous, so we assume $\phi \not\equiv 0$. 
To proceed, we need a lower bound for the operator $\mathscr{R}_E f$ under small perturbations.

\begin{claim}\label{claim.gateaux-bound}
There exists $\bar t > 0$ independent of $E$ such that
\begin{align*}
	\mathscr R_E (f + t \phi)(y) \geq \frac{\varpi}{2} \| y \|^d, \quad \text{ for every } y \in E, \, |t| \leq \bar t.
\end{align*}
\end{claim}

\textit{Proof of Claim~\ref{claim.gateaux-bound}.}
Thanks to Remark~\ref{rem:integrability}, there exists $\varpi > 0$ independent of $E$ such that $\mathscr R_E f(y) \geq \varpi \| y \|^d$ for every $y \in E$ and $E \in \Gr{j}{n}$.
For a nonzero $\phi \in \cH_{n, d}$, set $M = \textstyle \| \phi \|_{L^\infty(\mathbb S^{n - 1})} \in (0, + \infty)$.
Choose $\bar t = \varpi/2M$.
Thanks to the $d$--homogeneity of $\phi$, for any $|t| \leq \bar t$, we have
\[
	|t| |\phi(y)| = |t| \| y \|^d |\phi(y / \| y \|)| \leq M \bar t \| y \|^d, \quad \text{ for every } y \in E \setminus \{ 0 \}.
\]
Then, for any $|t| \leq \bar t$, a direct computation leads to
\begin{align*}
	\mathscr R_E (f + t \phi)(y) = \mathscr R_E f(y) + t \mathscr R_E \phi(y) \geq \varpi \| y \|^d - |t| |\phi(y)| \geq \frac{\varpi}{2} \| y \|^d, \quad \text{ for every } y \in E,
\end{align*}
which completes the proof. \hfill  $\lozenge$

On the one hand, note that
\[
	\dfrac{d}{dt} \exp(- \mathscr R_E(f + t \phi)(y)) = - \mathscr R_E \phi(y) \exp( - \mathscr R_E(f + t \phi)(y)), \quad \text{ for every } y \in  E.
\]
On the other hand, the $d$--homogeneity of $\phi$ implies $| \mathscr R_E \phi(y) | \leq M \| y \|^d$.
Hence, combining with Claim~\ref{claim.gateaux-bound}, we have
\begin{align*}
	\left\vert (\mathscr R_E\phi)(y)\,\exp(-\mathscr R_E (f + t \phi)(y)) \right\vert
\ \le\ M(1+\|y\|^d)\,e^{-(\varpi/2)\|y\|^d}, \quad \text{ for every } |t| \leq \bar t, \, y \in E.
\end{align*}
A direct computation leads to
\[
    \int_{G(j, n)} \int_E (1 + \| y \|^d) e^{-(\varpi/2)\|y\|^d} \, dy \, d\nu_j(E) < + \infty.
\]
Therefore, by the dominated convergence theorem, we get
\begin{align*}
       \mathrm d \widetilde\opV_j(f; \phi) = - \dfrac{\sigma_{j, n}}{\Gamma(1 + j/d)} \int_{G(j, n)}\int_E (\mathscr{R}_E \phi)(y) \exp(- \mathscr R_E f(y)) \, dy \, d\nu_j(E).
\end{align*}
Note that the linearity of $\phi \mapsto \mathrm d \widetilde\opV_j(f; \phi)$ follows from the linearity of the operator $\mathscr R_E$.
Moreover, one can directly check that
\begin{align*}
	|  \mathrm d \widetilde\opV_j(f; \phi)| \leq \dfrac{\sigma_{j, n}}{\Gamma(1 + j/d)}\| \phi \|_{L^\infty(\Sph^{n - 1})} \underbrace{\int_{\R^j} \| x \|^d \exp(- \varpi \| x \|^d) dx}_{\,  < \, + \infty},
\end{align*}
which leads to the boundedness of the map $\phi \mapsto \mathrm d \widetilde\opV_j(f; \phi)$.
Lastly, we can conclude that $\widetilde\opV_j$ is G\^ateaux differentiable, which finishes the proof of Lemma~\ref{lem.Vtil-gateaux}.
\end{proof}

\begin{lemma}[G\^ateaux derivative of $\opV_j$]\label{lem.V-gateaux}
Let $f\in\cP_{n,d}$ and let $1 \leq j \leq n - 1$.
Then, for each $E\in G(j,n)$ and $y\in E$, the problem
\begin{equation}\label{min.E^perp}
\min_{z\in E^\perp} f(y+z)
\end{equation}
has a unique minimizer $z^\ast(E,y) \in E^\perp$, characterized by the projected first-order condition
\begin{equation}\label{eq:proj-FOC}
P_{E^\perp}\,\nabla f\big(y+z^\ast(E,y)\big)=0
\qquad\text{(equivalently, $\nabla f(y+z^\ast(E,y))\in E$).}
\end{equation} 
Furthermore, for each $\phi \in \cH_{n, d}$, it holds
\[
\mathrm{d} \opV_j(f; \phi)
= -\frac{\beta_{j,n}}{\Gamma(1+j/d)}\int_{G(j,n)}\!\int_E \phi\big(y+z^\ast(E,y)\big)\, \exp(-\Pi_E f(y))\,dy\,d\nu_j(E).
\]
\end{lemma}

\begin{proof}
Observe first that $f$ is strictly convex.
Indeed, if not, there would exist $\bar x \neq \bar y$ and $\lambda \in (0, 1)$ such that
\[
  f(\lambda \bar y + (1 - \lambda) \bar x) = \lambda f(\bar y) + (1 - \lambda) f(\bar x).  
\]
Equality in Jensen's inequality happens when $f$ restricted to the interval $[\bar x, \bar y] := \big\{ \bar x + t (\bar y - \bar x): t \in [0, 1] \big\}$ is affine.
Set $h(t) := f(\bar x + t(\bar y - \bar x))$ for every $t \in \R$.
Since $h$ is a polynomial and the restriction of $h$ on $[0, 1]$ is affine, it is an affine function. 
Thanks to the homogeneity of $f$, we have
\begin{align*}
    h(t) = f(\bar x + t(\bar y - \bar x)) = t^d f(\bar y - \bar x + t^{-1}\bar x), \quad \text{ for every } t > 0,
\end{align*}
which leads to 
\[
    \lim_{t \to \infty} \dfrac{h(t)}{t^d} = f(\bar y - \bar x) > 0.
\]
Therefore, the leading coefficient of $h$ is $f(\bar x - \bar y) > 0$ and so it is a polynomial of degree $d \geq 2$, which contradicts the fact that $h$ is affine.

\medskip

The restriction $z\mapsto f(y+z)$ to the affine subspace $y+E^\perp$ is strictly convex and coercive, hence admits a unique minimizer $z^\ast(E,y)\in E^\perp$.
The first-order condition \eqref{eq:proj-FOC} follows directly from the basic optimality rule~\cite[Theorem 6.12]{RW_1998}: 
for the minimization of $f$ over the affine subspace 
$y+E^\perp$, the minimizer $z^\ast(E, y)$ satisfies
\[
0 \in \nabla f(y + z^\ast(E, y)) + N_{y + E^\perp}(y+ z^{\ast}(E, y)), 
\]
where $N_U(z)$  denote the normal cone of $U$ at $z$ in the sense \cite[Definition 6.3]{RW_1998}.
Since $N_{y+E^\perp}(w) = E$ for every $w \in y + E^\perp$, this yields 
\begin{align*}
    \text{$\nabla f(y+ z^*(E, y)) \in E$ or equivalently 
$P_{E^\perp}\nabla f(y+\bar z)=0$.}
\end{align*}
Furthermore, according to~\cite[Theorem 14.37]{RW_1998}, we observe that the map $(E, y) \mapsto z^\ast(E, y)$ is measurable.

\medskip

To proceed, we state an analogous argument as in Claim~\ref{claim.gateaux-bound} adapted for the operator $\Pi_E$, whose proof will be omitted.
As a direct consequence of the following claim, the function $f  + t \phi$ is coercive for every $-\bar t \leq t \leq \bar t$.

\begin{claim}\label{claim.gateaux-bound-Pi}
    There exists $\bar t > 0$ independent of $E$ such that for every $|t| \leq \bar t$ and $y \in E$,
    \begin{align*}
        \Pi_E(f + t \phi)(y) \geq \dfrac{\varpi}{2} \| y \|^d, \quad \text{ where } \varpi := \min_{\mathbb S^{n - 1}} f > 0.
    \end{align*}
\end{claim}

Fix $E \in \Gr{j}{n}$ and $y \in E$.
Set 
\begin{align*}
    g_t(z) := & ~ f(y + z) + t \phi(y + z), \quad \text{ for every } z \in E^\perp, \\
    \Pi(t) := & ~ \Pi_E(f + t\phi)(y) = \min_{z \in E^\perp} g_t(z), \\
    Z^\ast_t(E, y) := & \argmin_{z \in E^\perp} g_t(z).
\end{align*}
Notice that the function $\Pi: [-\bar t, \bar t] \to [0, + \infty)$ is  concave and hence, it is Lipschitz and by Rademacher theorem, is differentiable almost every $t$.
By Danskin theorem (see e.g \cite[Theorem 4.13]{BonnansShapiro2000}), for a.e $t$, it holds
\begin{equation}\label{Pi-prime}
    \Pi'(t) \in \mathrm{co} \big\{  \phi(y + z) ~:~ z \in Z_t^\ast(E, y) \big\},
\end{equation}
where $\mathrm{co}(A)$ denotes the convex hull of a set $A$.
At $t = 0$, we have observed that $Z_0^\ast(E, y) = \{ z^\ast(E, y) \}$ and consequently, we have
\begin{equation}\label{deriv-Pi(t)}
\left. \frac{d}{d t} \right|_{t = 0} e^{-\Pi(t)}
= - \ e^{-\Pi_E f(y)} \left. \frac{d}{d t} \right|_{t = 0}\Pi_E (f+ t\phi)(y)
= -\,\phi(y+z^\ast(E, y))\,e^{-\Pi_E f(y)}.
\end{equation}
For $t \neq 0$, the fundamental theorem of calculus gives
\begin{align*}
    \dfrac{e^{-\Pi(t)} - e^{-\Pi(0)}}{t} = - \int_0^1 \Pi'(\theta t) e^{- \Pi(\theta t)} d\theta.
\end{align*}
To justify the integrability of the above expresion, on the one hand, it follows from Claim~\ref{claim.gateaux-bound-Pi} that
\begin{align*}
    e^{- \Pi(\theta t)} = e^{- \Pi_E(f + \theta t \phi)(y)} \leq e^{- (\varpi/2) \| y \|^d}, \quad \text{ for every } |t| \leq \bar t \text{ and  } \theta \in [0, 1]. 
\end{align*}
On the other hand, thanks to \eqref{Pi-prime} and the homogeneity of $\phi$, we have
\begin{align*}
    |\Pi'(\theta t)| \leq \| \phi \|_{L^\infty(\mathbb S^{n - 1})} \| y \|^d, \quad \text{ for every } |t| \leq \bar t \text{ and } \theta \in [0, 1].
\end{align*}
Therefore, we obtain, for every $t \in [- \bar t, \bar t]$,
\begin{align*}
    \left\vert \dfrac{e^{-\Pi(t)} - e^{-\Pi(0)} }{t} \right\vert \leq \int_0^1 |\Pi'(\theta t)e^{-\Pi(\theta t)}| d\theta \leq \| \phi \|_{L^\infty(\mathbb S^{n - 1})} \| y \|^d e^{- (\varpi/2) \| y \|^d},
\end{align*}
in which the r.h.s is integrable in $E$ (uniformly in $E$).
In view of the above observations, applying the Lebesgue dominated convergence theorem and using~\eqref{deriv-Pi(t)}, we arrive at 
\begin{align*}
    \mathrm d\opV_j(f;\phi) = & ~ 
    \dfrac{\beta_{j, n}}{\Gamma(1 + j/d)} \lim_{t \to 0} \int_{\Gr{j}{n}} \int_E \dfrac{e^{- \Pi(t)} - e^{- \Pi(0)} }{t} \, dy \, d\nu_j(E) \\
    = & ~ 
    \dfrac{\beta_{j, n}}{\Gamma(1 + j/d)} \int_{G(j, n)} \int_E \lim_{t \to 0} \dfrac{e^{- \Pi(t)} - e^{- \Pi(0)} }{t} \, dy \, d\nu_j(E) \\
    = & ~ - \int_{\Gr{j}{n}} \int_E \phi\big(y+z^\ast(E,y)\big)\, \exp(-\Pi_E f(y)) \, dy \, d\nu_j(E),
\end{align*}
which completes the proof.
\end{proof}

Now, we are able to characterize the first--order condition for the generalized L\"owner--John ellipsoid associated with intrinsic and dual volumes. 
Thanks to the G\^ateaux differentiability in Lemmas \ref{lem.V-gateaux}--\ref{lem.Vtil-gateaux} together with KKT condition, the first--order condition for the minimizers of \eqref{P-0} and \eqref{Ptil-0} is stated in the following proposition.

\medskip

We state the result only for \eqref{P-0}; an analogous result also applies to \eqref{Ptil-0}. 
The proof of Proposition~\ref{prop.KKT-JOHN} is similar to that of \cite[Theorem 3.2]{La_2015}.
\medskip


\begin{proposition}\label{prop.KKT-JOHN}
Let $1 \leq j \leq n - 1$ and let $K$ be a compact set containing the origin in its interior.
Denote $\{ \phi_k \}_{k = 1}^N$ the canonical basis of $\cH_{n, d}$.
Then, problem~\eqref{P-0} admits a unique minimizer $f^\star \in \cP_{n, d}$ and the following assertions hold true:
\begin{itemize}
	\item[(i)] There exists a finite nonnegative Borel measure $\mu^\star$ on $K$ such that, for every $k \in \{1, \cdots, N \}$,
	\begin{align*}
		& \frac{\beta_{j,n}}{\Gamma(1+j/d)}\int_{G(j,n)}\!\int_E \phi_k \big(y+z^\ast(E,y)\big)\, \exp(-\Pi_E f^\star(y))\,dy\,d\nu_j(E) = \int_K \phi_k(x) d\mu^*(x),  \\
		& \int_K (1 - f^\star(x))d\mu^\star(x) = 0.
	\end{align*}
	Consequently, $\mathrm{supp} \mu^\star \subset [f^\star = 1]$ and $\mu^\star(K) = (j/d) V_j([f^\star \leq 1])$.
	Furthermore, the measure $\mu^\star$ can be chosen atomic, i.e,
	\begin{align*}
		\mu^* = \sum_{\ell = 1}^M \lambda_\ell \delta_{x^\ell}, \quad 
		\text{ where } \lambda_\ell > 0 \text{ and } x^\ell \in K \cap [f^\star = 1],
	\end{align*}
	and the optimality condition becomes
	\begin{align*}
		\frac{\beta_{j,n}}{\Gamma(1+j/d)}\int_{G(j,n)}\!\int_E \phi_k \big(y+z^\ast(E,y)\big)\, \exp(-\Pi_E f^\star(y))\,dy\,d\nu_j(E) = \sum_{\ell = 1}^M \lambda_\ell \phi_k(x^\ell), 
	\end{align*}
for every $k \in \{1, \cdots, N \}$ and $f^\star(x^\ell) = 1$ for every $\ell \in \{1, \cdots, M \}$.
	\item[(ii)]
	Let $f^\star\in P_{n,d}$ be feasible.
	Assume that there exist finite points $\{ (\lambda_\ell, x^\ell) \} \subset (0, + \infty) \times K$ satisfying $f^\star(x^\ell) = 1$ for every $\ell$.
If one has
	\begin{align*}
	\frac{\beta_{j,n}}{\Gamma(1+j/d)}\int_{G(j,n)}\!\int_E \phi_k \big(y+z^\ast(E,y)\big)\, \exp(-\Pi_E f^\star(y))\,dy\,d\nu_j(E) = \sum_{\ell = 1}^M \lambda_\ell \phi_k(x^\ell), \,\, \text{ for every } k,
	\end{align*}
then $f^\star$ is the unique minimizer of~\eqref{P-0}.
\end{itemize}
\end{proposition}

\subsection{Nonnegative polynomial with sublevel sets of minimal intrinsic/dual volumes}
Let $\tnorm{\cdot}$ be a norm in $\cH_{n, d}$.
This section deals with the following optimization problems:
\begin{equation}\label{P-unit-bal}\tag{$Q_0$}
	\text{minimize} \quad V_j([f \leq 1]) \quad \text{ such that } \quad f \in \cP_{n, d}, \,\, \tnorm{f} \leq 1,
\end{equation}
and
\begin{equation}\label{Ptil-unit-ball}\tag{$\widetilde Q_0$}
	\text{minimize} \quad \widetilde V_j([f \leq 1]) \quad \text{ such that } \quad f \in \cP_{n, d}, \,\, \tnorm{f} \leq 1.
\end{equation}
A related problem has been studied in the context of the volume functional, as shown in \cite{KL_2022}, in which the authors characterized the minimizer for various $O(n)$–invariant norms on $\cH_{n, d}$.
The problems \eqref{P-unit-bal}--\eqref{Ptil-unit-ball} extend to the case of intrinsic and dual volumes.
Note that the above problem is simply optimizing the functional $\opV_j$ and $\widetilde \opV_j$ over the intersection of $\cP_{n, d}$ and the unit ball in $(\cH_{n, d}, \tnorm{\cdot})$.

\medskip

Thanks to the strict log--convexity in Corollary~\ref{cor:strict-logconvex} and the lower semicontinuity in Corollary~\ref{cor.lsc}, it is straightforward to see that the problems~\eqref{P-unit-bal} and ~\eqref{Ptil-unit-ball} admit a unique minimizer. 
Furthermore, using a similar argument as those in \cite[Theorem 1.2]{KL_2022}, we have an exact minimizer whenever the norm $\tnorm{\cdot}$ is $O(n)$--invariant. 
We summarize these facts in the following proposition.
This shows that the Euclidean ball is the smallest intrinsic and dual volumes among all convex homogeneous polynomials with bounded $O(n)$--invariant norm.

\begin{proposition}
Let $1 \leq j \leq n$ and let $\tnorm{\cdot}$ be a norm in $\cH_{n, d}$.
Then, the following assertions hold true:
\begin{itemize}
    \item[(i)] The problems \eqref{P-unit-bal} (resp. \eqref{Ptil-unit-ball}) admits a unique minimizer $f^\star$ (resp. $\widetilde f^\star$).
    \item[(ii)] If $\tnorm{\cdot}$ is $O(n)$--invariant, that is,
    \begin{align*}
        \tnorm{f \circ \rho} = \tnorm{f} \quad \text{ for every } f \in \cH_{n, d} \text{ and } \rho \in O(n),
    \end{align*}
    then $f^\star = \widetilde f^\star = b_{n, d}/\tnorm{b_{n, d}}$, where $b_{n, d}(x) = (x_1^2 + \cdots + x_n^2)^{d/2}$.
    \end{itemize}
\end{proposition}

\subsection{Computing intrinsic volumes via block--decomposition}
Beyond the optimization applications mentioned earlier, we now turn to a quantitative feature of intrinsic and dual volumes.  
Block decompositions offer a way to simplify the computation of  dual and intrinsic volumes.
When a function splits orthogonally into components supported on mutually orthogonal subspaces, its projections and sections inherit a separable structure. 

\medskip

Let $\{ U_b \}_{b \in \{1, \cdots, B\}}$ be orthogonal subspaces of $\R^n$ such that
\[
\R^n = \bigoplus_{b=1}^B U_b
\]
which is called an orthogonal decomposition.
Assume that $f$ is \textit{block--separable} (relative to $\{ U_b \}$), i.e, there exist $f_b \in \cP_{U_b, d}$ for every $b \in \{1, \cdots, B \}$ such that  
\[
f(x) = \sum_{b=1}^B f_b(P_{U_b}x), \quad \text{ for every } x \in \R^n.
\]
Here, for each block $U_b$, we denote $P_{U_b}: \R^n \longrightarrow U_b$ the orthogonal projection onto $U_b$ and we also denote
\[
\cP_{U_b,d} :=  \big\{ f: U_b \to [0, +\infty) \;:\; f \text{ is $d$--homogeneous, convex and positive except the origin} \big\}.
\]
The following proposition provides an efficient way to compute intrinsic volumes of sublevel sets of block-separable homogeneous functions

\begin{proposition}[Block-orthogonal decomposition]\label{prop:block-separable}
Let $\{ b_i \}_{i \in \{1, \cdots, r \}} \subset \{1, \cdots, B \}$ and let $E\subset U_{b_1}\oplus\cdots\oplus U_{b_r}$ be any subspace orthogonal to the remaining blocks. 
Then, if one writes
\[
y=\sum_{b\in\{b_1,\dots,b_r\}} y_b \quad \text{ where $y_b\in U_b$ for $b \in \{b_1, \cdots, b_r \}$ },
\]
then it holds
\[
\Pi_E f(y)=\sum_{b\in\{b_1,\dots,b_r\}} \Pi_{E\cap U_b}\,f_b(y_b)
\qquad\text{and}\qquad
\mathscr R_E f(y) = \sum_{b \in \{b_1, \cdots, b_r\}} \mathscr R_{E\cap U_b} f_b(y_b).
\]
In particular, the exponential integrals factor:
\begin{equation}\label{id_Pi_block}
\int_E \exp(-\Pi_E f(y))\,dy=\prod_{b\in\{b_1,\dots,b_r\}}\ \int_{E\cap U_b} \exp(-\Pi_{E\cap U_b} f_b(y_b)) \,dy_b,
\end{equation}
and analogous
\begin{equation}
    \int_E \exp(- \mathscr R_E f(y)) \, dy = \prod_{b \in \{b_1, \cdots, b_r \}} \int_{E \cap U_b} \exp(- \mathscr R_{E \cap U_b} f_b(y_b)) \, dy_b.
\end{equation}
\end{proposition}

\begin{remark}\normalfont
As a direct consequence of Proposition~\ref{prop:block-separable},  
for any block-separable function $f \in \mathcal P_{n,d}$ and for any $1 \le j \le n-1$,  
the intrinsic and dual intrinsic volumes of its sublevel set $[f \le \alpha]$ admit respectively the following factorized representations:
\begin{align*}
    V_j([f \leq \alpha]) = \dfrac{\alpha^{j/d}\beta_{j, n}}{\Gamma(1 + j/d)} \int_{G(j, n)} \prod_{b \in B_E} \int_{E \cap U_b} \exp(- \Pi_{E \cap U_b} f_b(y_b)) \, dy_b \, d\nu_j(E),
\end{align*}
and 
\begin{align*}
    \widetilde V_j([f \leq \alpha]) = \dfrac{\alpha^{j/d}\sigma_{j, n}}{\Gamma(1 + j/d)} \int_{G(j, n)} \prod_{b \in B_E} \int_{E \cap U_b} \exp(- \mathscr R_{E \cap U_b} f_b(y_b)) \, dy_b \, d\nu_j(E),
\end{align*}
where the active index set $B_E$ is defined by
\[
    B_E := \{ b \in \{1, \cdots, B \}: U_b \cap E \neq \{ 0 \} \}.
\]
\end{remark}

\begin{proof}[Proof of Proposition~\ref{prop:block-separable}]
Using the fact that $E$ is orthogonal to $U_b$ for $b \not\in \{b_1, \cdots, b_r\}$, it follows from the definition of the family $\{ U_b \}$ that
\begin{equation}\label{E^perp_decomp}
    E^\perp 
    =
    \left( \bigoplus_{b\in\{b_1,\dots,b_r\}} \big( U_b\cap E^\perp \big) \right)
    \,\oplus\,
    \left( \bigoplus_{b\notin\{b_1,\dots,b_r\}}U_b \right).
\end{equation}
Fix 
\begin{align*}
    y = & \sum_{b\in\{b_1,\dots,b_r\}}y_b\in E, \quad  \text{ where } b \in \{b_1, \cdots, b_r \}, \\
    z = & \sum_b z_b \in E^\perp, \quad \text{ where }
    z_b \in 
    \begin{dcases}
        U_b \cap E^\perp, & \text{ if } b \in \{b_1, \cdots, b_r \},\\
        U_b, & \text{ otherwise}
    \end{dcases}.
\end{align*}
Using the fact that $f$ is block--separable, we have
\begin{equation}\label{f_decomp}
f(y+z)=\sum_{b \in \{b_1, \cdots, b_r \}} f_b(y_b+z_b) + \sum_{b \not\in \{b_1, \cdots, b_r \}} f_b(z_b).
\end{equation}
It follows from \eqref{E^perp_decomp} and \eqref{f_decomp} that for any $y \in E$,
\begin{align*}
\Pi_E f(y)
= & ~ 
\inf_{z\in E^\perp}\sum_{b \in \{b_1, \cdots, b_r \}} f_b(y_b + z_b) + \sum_{b \not\in \{b_1, \cdots, b_r\}} f_b(z_b)  \\
= & ~ 
\sum_{b\in\{b_1,\dots,b_r\}}\inf_{z_b\in U_b\cap E^\perp} f_b(y_b+z_b)\ +\ \sum_{b\notin\{b_1,\dots,b_r\}} \inf_{\hat z_b\in U_b} f_b(\hat z_b).
\end{align*}
Notice that $\textstyle\inf_{\hat z_b \in U_b} f_b(\hat z_b) = 0$.
This yields the projection identity:
\begin{align*}
    \Pi_E f(y) = \sum_{b \in \{b_1, \cdots, b_r \}} \Pi_{E \cap U_b} f_b(y_b).
\end{align*}
Finally, one can use Fubini theorem and orthogonality of the splitting 
\[
E=\bigoplus_{b\in\{b_1,\dots,b_r\}} (E\cap U_b)
\]
to obtain the identity~\eqref{id_Pi_block}.
Lastly, the linearity of $\mathscr{R}_E$ directly implies that
\begin{align*}
    \mathscr R_E f(y) = \mathscr R_E \sum_{b \in \{b_1, \cdots, b_r \}} f_b(y_b) = \sum_{b \in \{b_1, \cdots b_r\} } \mathscr R_{E \cap U_b} f_b(y_b).
\end{align*}
One can proceed similarly as the proof for the projection operator $\Pi_E$ to obtain the result for $\mathscr R_E$.
Proposition~\ref{prop:block-separable} is proven.
\end{proof}

\begin{example}\normalfont
Let us consider $\R^n = U_1 \oplus U_2$, where $U_1 = \R^m$ and $U_2 = \R^{n - m}$.
Let us fix $a, b > 0$.
Consider $f(x) = a \| x' \|^d + b \| x'' \|^d$ where $x = (x', x'') \in U_1 \times U_2$.
For any $E \in \Gr{j}{n}$, set $i_E = \mathrm{dim} (E \cap U_1)$ and $j - i_E = \mathrm{dim}(E \cap U_2)$.
It follows from Proposition~\ref{prop:block-separable} that
\begin{align*}
    \int_E \exp(- \Pi_E f(y)) \, dy 
    = & ~ 
    \int_{E \cap U_1} \exp(- a \| y \|^d) \, dy \int_{E \cap U_2} \exp(- b\| y \|^d) \, dy \\
    = & ~
    \kappa_{i_E} \kappa_{j - i_E}a^{-i_E/d}b^{-(j - i_E)/d}\Gamma(1 + i_E/d)\Gamma(1 + (j - i_E)/d).
\end{align*}
Thanks to \cite[Lemma 13.2.1]{SW_2008}, we have that 
\[
i_E = i_\star = \max \{ 0, j - (n - m) \} \quad \text{ for $\nu_j$--a.e $E$.}
\]
Therefore, for any $1 \leq j \leq n - 1$, we get an explicit formula for the intrinsic volume of the sublevel set of $f$:
\begin{align*}
    V_j([f \leq \alpha]) = \dfrac{\alpha^{j/d}\beta_{j, n}}{\Gamma(1 + j/d)} \kappa_{i_\star} \kappa_{j - i_\star} a^{-i_\star/d} b^{-(j - i_\star)/d} \Gamma(1 + i_\star/d) \Gamma(1 + (j - i_\star)/d).
\end{align*}
\end{example}

\section{Arithmetic applications of the exponential representations}
\label{sec:arith-applications}

This section discusses arithmetic applications of the Laplace–Grassmannian representation, focusing on how intrinsic volumes govern lattice–point discrepancies in convex polynomial sublevel sets. 
In particular, we explore Lipschitz-type bounds and related counting results linking these analytic representations with classical problems in the geometry of numbers.

\medskip


\subsection{A Lipschitz–type lattice discrepancy bound via intrinsic volumes}\label{sec:lipschitz}

In what follows, we quantify the discrepancy between lattice points and the volume of sublevel sets.
For lattice points in the region ${[f \le \alpha]}$ (and its variants), the leading term is of order $\alpha^{n/d}$, while the error term has order $\alpha^{(n-1)/d}$. 
For a convex body $K \subset \R^n$, the classical heuristic 
\[
\#(K \cap \Z^n) \approx \vol_n(K)
\]
suggests that the discrepancy is controlled by the size of the boundary—an idea dating back to 
Davenport’s Lipschitz principle~\cite{Davenport1951, Davenport1964Corr} and the mean and second-moment 
bounds of Rogers~\cite{Rogers1955, Rogers1956}. 
To our knowledge, however, the explicit use of the projection operator $\Pi_E f$ and the section operator 
$\mathscr{R}_E f$ to parameterize the error constants in lattice problems has not appeared in 
the literature. 

\begin{lemma}\label{lem:boundary-cubes}
Let $K\subset\R^n$ be a convex body.
For any fixed $m \in \mathbb Z^n$, let $T_m := m + (-\tfrac12,\tfrac12]^n$ be the half–open unit cubes partitioning $\R^n$.
Set
\[
  \mathcal B(K) := \left\{ m \in \Z^n:\ T_m\cap K\neq\varnothing\text{ and }T_m\cap K^c\neq\varnothing\,\right\}
\]
the set of boundary cubes.
Then, it holds
\begin{equation}
  \big|\#(K\cap\Z^n)-\vol_n(K)\big|\ \le\ |\mathcal B(K)|.
\end{equation}
Furthermore, one has
\begin{equation}\label{discre_bounda}
  |\mathcal B(K)|\ \le\ \vol_n\big(\{x\in\R^n:\ \operatorname{distance}(x,\partial K)\le \sqrt n\}\big).
\end{equation}
\end{lemma}

\begin{proof}
Observe first that
\[
  \vol_n(K)\ =\ \sum_{m\in\Z^n} \vol_n\big(T_m\cap K\big) \quad \text{ and } \quad
  \#(K\cap\Z^n)\ =\ \sum_{m\in\Z^n} \pmb 1_K(m).
\]
If $T_m\subset K$, then $\vol_n(T_m \cap K) = \vol_n(T_m) = 1$ and $\pmb 1_K(m) = 1$.
If $T_m\cap K=\varnothing$, then they both vanish, i.e, $\vol_n(T_m \cap K) = \pmb 1_K(m) = 0$.
Thus, we have
\[
  \#(K\cap\Z^n)-\vol_n(K)\ =\ \sum_{m\in\mathcal B(K)}
  \big( \pmb 1_K(m)-\vol_n(T_m\cap K) \big).
\]
Since $\pmb 1_K(m) \in \{ 0, 1 \}$ and $0 \leq \vol_n(T_m \cap K) \leq 1$, we directly get $|\pmb 1_K(m) - \vol_n(T_m \cap K) | \leq 1$. 
This implies that
\begin{align*}
    \big\vert \#(K \cap \Z^n) - \vol_n(K)  \big\vert \leq \big\vert \mathcal B(K) \big\vert.
\end{align*}

\medskip

It remains to check the inequality \eqref{discre_bounda}.
Fix $m\in\mathcal B(K)$.
By the definition of $\mathcal B(K)$, there exist $x\in T_m\cap K$
and $y\in T_m\cap K^c$. 
Since $T_m$ is convex and $K$ is a convex body, the segment $[x,y]\subset T_m$ intersects $\partial K$ at some point~$z$.
Consequently, we obtain the following estimate
\[
  \operatorname{distance}(w,\partial K)  \leq \|w-z\| \leq \operatorname{diameter}(T_m) = \sqrt n, \quad \text{ for any $w\in T_m$}.
\]
Hence $T_m\subset \{x:\operatorname{distance}(x,\partial K)\le\sqrt n\}$. 
Since the cubes $T_m$ are pairwise disjoint and each has volume $1$, we infer that
\[
  |\mathcal B(K)| = \sum_{m\in\mathcal B(K)}\vol_n(T_m)
  = \vol_n \left(\bigcup_{m\in\mathcal B(K)}T_m\right)
  \leq \vol_n\big(\{x \in \R^n :\operatorname{distance}(x,\partial K)\le\sqrt n\}\big),
\]
which completes the proof.
\end{proof}

\begin{lemma}[Discrepancy via intrinsic volumes]\label{prop:disc-V}
There exists a constant $C_n>0$ depending only on $n$ such that for every convex body $K\subset\R^n$,
\begin{equation}\label{eq:disc-intrinsic}
  \big|\#(K\cap\Z^n)-\vol_n(K)\big|\ \le\
  C_n\,\sum_{j=0}^{n-1} V_j(K).
\end{equation}

Moreover, if $K$ ranges in a one-parameter homothetic family $\{ K_t = tK_0 \}_{t > 0}$ with fixed $K_0$, then there exists a constant $C_{K_0}>0$ such that for all $t>0$,
\begin{equation}\label{eq:disc-intrinsic-(n-1)}
  \big|\#(K_t\cap\Z^n)-\vol_n(K_t)\big|\ \le\ C_nC_{K_0} \big(1+V_{n-1}(K_t)\big).
\end{equation}
\end{lemma}

\begin{remark}
In general, the dependence on lower-order terms in the asymptotic behavior~\eqref{eq:Nf-asymp-lowe-term} cannot be suppressed. 
A partial reason follows from an observation that the sequence of intrinsic volumes \( \{ V_j(K) \} \) is not monotone up to a dimensional constant; that is, there does not exist a constant \( C_n > 0 \) such that $V_i(K) \leq C_n V_j(K)$ for every convex body \( K \) and all indices \( i < j \). 
To see this, consider the following family of convex bodies:
\[ P_\varepsilon := [0,1]^i \times [0,\varepsilon]^{n - i} 
\]
for some fixed \( 1 < i < n \). 
According to the representation theorem for valuations on parallelotopes~\cite[Theorem 4.2.1--4.2.2]{KR_1997}, the intrinsic volumes of \( P_\varepsilon \) can be explicitly computed as follows:
\begin{equation}\label{V_j=e_j}
    V_j(P_\varepsilon) = c_{j, n} e_j(\underbrace{1, \cdots, 1}_{i}, \underbrace{\varepsilon, \cdots, \varepsilon}_{n - i}), \quad \text{ for every } 0 \leq j \leq n,
\end{equation}
where $c_{j, n} > 0$ is a normalized constant and $e_j$ is the elementary symmetric function; more precisely, $e_0 \equiv 1$ and for $j \geq 1$,
\begin{align*}
    e_j(x_1, \cdots, x_n) = \sum_{1 \leq \ell_1 < \cdots < \ell_j \leq n} x_{\ell_1}x_{\ell_2} \cdots x_{\ell_j}.
\end{align*}
On the one hand, in view of the formula~\eqref{V_j=e_j}, for any $j > i$ and for any $\varepsilon \in (0, 1)$, there exists a constant $C_{n, i, j}$ such that $V_j(P_\varepsilon) \leq C_{n, i, j} \varepsilon^{j - i}$.
On the other hand, since we have
\begin{align*}
    e_i(1, \cdots, 1, \varepsilon, \cdots, \varepsilon) = \sum_{r = 0}^i {i \choose r}{n - i \choose i - r} \varepsilon^{i - r} = 1 + \sum_{r = 0}^{i - 1} {i \choose r}{n - i \choose i - r} \varepsilon^{i - r},
\end{align*}
we observe that $V_i(P_\varepsilon) = c_{i, n} + O(\varepsilon)$ stays bounded from below by a positive constant as $\varepsilon \searrow 0$.
Therefore, we infer that for any fixed $j > i$, there exists a constant $A_{n, i, j} > 0$ such that
\begin{align*}
    \dfrac{V_i(P_\varepsilon)}{V_j(P_\varepsilon)} \geq A_{n, i, j} \varepsilon^{-(j - i)} \to \infty \quad \text{ as } \varepsilon \searrow 0.  
\end{align*}
This implies the desired conclusion.
\end{remark}

\begin{proof}[Proof of Lemma~\ref{prop:disc-V}]
Set $N_\rho(\partial K):=\{x \in \R^n:\operatorname{distance}(x,\partial K)\le\rho\}$.
In view of Lemma~\ref{lem:boundary-cubes}, we have
\[
  \big|\#(K\cap\Z^n)-\vol_n(K)\big|\ \le\ \vol_n\big(N_{\sqrt n}(\partial K)\big).
\]
For a convex $K$ and any $\rho>0$, one has the inclusion
$N_\rho(\partial K)\subset (K + \rho B^n)\setminus (K\ominus \rho B^n)$,
where $A + B$ and $A\ominus B := \bigcap_{a \in A} (B - a)$ denote Minkowski sum and (inner) erosion, respectively.
Then, applying Steiner formula and note that $\vol_n\big(K\ominus \rho B^n\big)
\ =\ \sum_{j=0}^{n}(-1)^{n-j}\kappa_{n-j}V_j(K)\rho^{n-j},$ we obtain
\begin{align*}
  \vol_n\big(N_{\sqrt n}(\partial K)\big)
 \leq & ~ \vol_n((K + \sqrt n\,B^n) \setminus (K \ominus \sqrt{n}B^n) \\
 = & ~ \vol_n(K + \sqrt{n}B^n) - \vol_n(K \ominus \sqrt{n}B^n) \\
 \leq & ~ C_n'\sum_{j=0}^{n-1}\kappa_{n-j}\,V_j(K)\,\sqrt{n}^{\,n-j}.
\end{align*}
(i.e., the term $V_n(K)$ vanishes), then absorbing constants gives ~\eqref{eq:disc-intrinsic}.

\medskip

To prove the estimate~\eqref{eq:disc-intrinsic-(n-1)}, by the homogeneity of intrinsic volumes, we obtain
\[
    \sum_{j=0}^{n-1}V_j(K_t)=\sum_{j=0}^{n-1}t^{\,j}V_j(K_0)
    \leq \sum_{j=0}^{n-1}V_j(K_0)(1+t^{\,n-1})
    \leq  C_{K_0}\,\bigl(1+t^{\,n-1}V_{n-1}(K_0)\bigr),
\]
where 
\[
C_{K_0} := \max \left\{\sum_{j=0}^{n-1}V_j(K_0),\ \frac{\sum_{j=0}^{n-1}V_j(K_0)}{V_{n-1}(K_0)} \right\}.
\]
Combining the above observations and \eqref{eq:disc-intrinsic}, we deduce the estimate \eqref{eq:disc-intrinsic-(n-1)}, which completes the proof.
\end{proof}

The right–hand side of~\eqref{eq:disc-intrinsic} is measured purely through intrinsic volumes, which fits perfectly with the exponential representations in Theorem~\ref{thm.representation}.
Now, for any $\alpha>0$ we define
\[
  N_f(\alpha):=\#\{x\in\Z^n:\ f(x)\le \alpha\}.
\]

\begin{proposition}[Asymptotic with explicit error]\label{thm:lattice-asymp}
Let $f \in \cP_{n, d}$.
Then, for all $\alpha > 0$, it holds
\begin{equation}\label{eq:Nf-asymp-lowe-term}
  N_f(\alpha) = \opV_n(f) \alpha^{n/d} + 
  O_n \left(  \sum_{j=0}^{n-1} \opV_j(f) \alpha^{j/d}  \right).
\end{equation}
In particular, incorporating the dependence on $f$, it holds
\begin{equation}\label{eq:Nf-asymp-higher-term}
    N_f(\alpha) = \opV_n(f) \alpha^{n/d} + O_{n, f}(\alpha^{(n - 1)/d}), \quad \text{ as } \alpha \to \infty.
\end{equation}
\end{proposition}

\begin{remark}\normalfont
(i)
The bound in Theorem~\ref{thm:lattice-asymp} matches the classical growth rate $\alpha^{(n-1)/d}$ (as in Davenport-type results). 
The contribution here is to express the constants in a coordinate-free way via intrinsic volumes of the base set ${[f \leq 1]}$, namely through $\opV_j(f)$. 
This makes the dependence on the shape of $[f\le 1]$ explicit and provides uniform control for even small $\alpha$ with the lower order terms. 
For any fixed $f$, the behavior reduces to ${O_{n,f}(\alpha^{(n-1)/d})}$.
We do not claim an improvement in the exponent; the point is a cleaner formulation via the use of Laplace--Grassmannian representations that can be convenient for other variants.

\medskip

(ii)
We impose no curvature assumptions on the boundary of $[f\le1]$.
Stronger error terms are known for smooth strictly convex bodies via oscillatory–integral methods (see, e.g.~\cite{Huxley1996, IvicKratzelKuehleitnerNowak2006, Kratzel1989}); the order of remainder $O_{n, f}(\textstyle\alpha^{(n-1)/d})$ is dimensionally sharp and robust but typically weaker than the best smooth/curved bounds.  
\end{remark}

\begin{proof}[Proof of Proposition~\ref{thm:lattice-asymp}]
Thanks to Lemma~\ref{prop:disc-V}, we know that
\begin{equation}\label{eq:disc-sumVj}
  \bigl| N_f(\alpha)-\vol_n([f \leq \alpha]) \bigr|
  \le\ O_n \left( \sum_{j=0}^{n-1} V_j([f \leq \alpha]) \right).
\end{equation}
Applying Theorem~\ref{thm.representation}, we get $V_j([f \leq \alpha]) = \alpha^{j/d} \opV_j(f)$.
Therefore, we obtain
\[
  \sum_{j=0}^{n-1} V_j([f \leq \alpha])
 = \sum_{j = 0}^{n - 1} \opV_j(f) \alpha^{j/d}.
\]
The above observations imply the estimate ~\eqref{eq:Nf-asymp-lowe-term}.
Furthermore, since $\alpha^{j/d} \leq \alpha^{(n - 1)/d}$ for every $\alpha > 1$ and $j \leq n - 1$, the estimate~\eqref{eq:Nf-asymp-higher-term} follows.
Proposition~\ref{thm:lattice-asymp} is proven.
\end{proof}




Let us continue with the study of primitive asymptotic results.
Denote
\[
\Z^n_{\rm prim}:=\{x\in\Z^n:\ \gcd(x_1,\dots,x_n)=1\} \quad \text{ and } \quad  N^{\rm prim}_f(\alpha):=\#\{x\in\Z^n_{\rm prim}:\ f(x)\le \alpha\}.
\]
Recall that the M\"obius function and the Riemann zeta function
are respectively defined by
\[
\mu(q) =
    \begin{dcases}
        1, & \text{if } q = 1\\
        (-1)^k, & \text{if } q \text{ is the product of } k \text{ distinct primes}\\
        0, & \text{if } q \text{ is divisible by the square of a prime}
    \end{dcases}
\]

and
\[
\zeta(s) := \sum_{q=1}^{\infty} \frac{1}{q^s} \quad \text{ if } \quad \Re(s) > 1.
\]
A standard M\"obius–inversion argument (applied to nonzero vectors) gives, for every $\alpha>0$,
\begin{equation}\label{eq:mobius-identity}
  N^{\rm prim}_f(\alpha)
  =
  \sum_{q=1}^{\infty}\mu(q) \,N_f^\ast\!\left(\frac{\alpha}{q^d}\right) \quad \text{ with } \quad N_f^\ast(\alpha):=N_f(\alpha)-1,
\end{equation}
where the sum is actually finite since $N_f^\ast(\alpha/q^d)=0$ once $q^d>\alpha/\tau$ with $\tau:=\textstyle\min_{x\in\Z^n\setminus\{0\}}f(x)>0$, see Proposition~\ref{prop.cF-coer}.
In the study of primitive asymptotics, it is worth noting that Lipschitz parameterizations yield the main term $\zeta(n)^{-1}\vol_n([f\le \alpha])$ together with a boundary–controlled error term, see e.g.~\cite{Widmer2010Primitive,Widmer2012Lipschitz}.
In what follows, by expressing the expansion in terms of intrinsic volumes, we obtain the correct geometric scaling and a consistent hierarchy of error terms.

\begin{proposition}[Primitive asymptotic]\label{thm:primitive}
Let $f\in\cP_{n,d}$. 
If $n\ge3$, then as $\alpha$ tend to $\infty$, it holds
\[
  N^{\rm prim}_f(\alpha)\ =\ \frac{\opV_n(f)}{\zeta(n)}\,\alpha^{n/d}
  \ +\ O_n\!\Big(\opV_{n-1}(f)\,\alpha^{(n-1)/d}\Big).
\]
In the case $n=2$, the same asymptotic expansion holds with the error $O\!\big(\opV_{1}(f)\,\alpha^{1/d}\log \alpha\big)$.
\end{proposition}

\begin{proof}
Applying the asymptotic behavior~\eqref{eq:Nf-asymp-lowe-term} in Proposition~\ref{thm:lattice-asymp}, we first observe that
\[
  N_f\!\left(\frac{\alpha}{q^d}\right)
  = 
  \opV_n(f)\,\alpha^{n/d} q^{-n}
  +
  O_n \left(
  \sum_{j=0}^{n-1}\opV_j(f)\,\alpha^{j/d} q^{-j}
  \right).
\]
Consequently, in view of the identity~\eqref{eq:mobius-identity}, we obtain
\begin{equation}\label{prim.full.asym}
  N^{\rm prim}_f(\alpha)
  = 
  \opV_n(f)\,\alpha^{n/d}\!\sum_{q\ge1}\frac{\mu(q)}{q^{n}}
  + O_n\!
  \left(\sum_{j=1}^{n-1}\opV_j(f)\,\alpha^{j/d}\!\sum_{q\ge1}\frac{|\mu(q)|}{q^{j}}
    + \sum_{q\le c\,\alpha^{1/d}}|\mu(q)|
    \right).
\end{equation}
Here we have used $N_f^{\ast} = N_f - 1$ to remove the constant term corresponding to $j = 0$. 
The upper bound $q \le c \textstyle\alpha^{1/d}$ reflects the fact that the M\"obius sum in~\eqref{eq:mobius-identity} is a finite sum, since 
$N_f^{\ast}(\alpha/q^d) = 0$ whenever $q^d > \alpha / \tau$. 
The remaining series in $q$ are absolutely convergent for $j \ge 2$ and in particular
\[
\sum_{q \ge 1} \frac{\mu(q)}{q^n} = \frac{1}{\zeta(n)}.
\]
In case $n \geq 3$,  following the above observations, the asymptotic expansion~\eqref{prim.full.asym} as $\alpha \to \infty$ becomes
\[
  N_f^{\mathrm{prim}}(\alpha)
  = \frac{\opV_n(f)}{\zeta(n)} \alpha^{n/d}
  + O_n 
  \big(\opV_{n-1}(f) \alpha^{(n-1)/d} \big).
\]
When $n=2$, the situation is slightly different: 
in this case, the sum $\textstyle\sum_{q\ge1}|\mu(q)|/q^j$ with $j=1$ does not converge but grows like $\log Q$ when truncated at $q\le Q$. 
Since our M\"obius sum effectively stops at $q \le c \textstyle\alpha^{1/d}$, 
this leads to an additional factor $\log \alpha$ in the remainder term. 
As a consequence, the error becomes
\[
O\!\big(\opV_1(f)\,\alpha^{1/d}\log\alpha\big).
\]
Proposition~\ref{thm:primitive} is proven.
\end{proof}

\begin{remark}
A straightforward consequence in arithmetic geometry is the following.  
Fix $n \ge 3$ and $f \in \cP_{n,d}$.
Define the Archimedean $f$–height on $\mathbb{A}^n(\mathbb{Q})$ by 
$H_f(x) := f(x)^{1/d}$, which is an analogue of the height considered in~\cite{ChambertLoirTschinkel2010}.  
Counting rational points on projective space via primitive representatives in a fixed orthant, we have
\[
  \# \big\{[x]  \in \mathbb{P}^{n-1}(\mathbb{Q}) : H_f(x) \le \beta \big\}
  = \frac{1}{2^n}\, N^{\mathrm{prim}}_f(\beta^d).
\]
Applying Proposition~\ref{thm:primitive} with $\alpha = \beta^d$ gives
\[
  \# \big\{ 
  [x] \in \mathbb{P}^{n-1}(\mathbb{Q}) : H_f(x) \le \beta
  \big\}
  = \frac{\opV_n(f)}{\zeta(n)}\, \beta^{n}
  + O_n\!\big(\opV_{n-1}(f)\, \beta^{n-1}\big)
  \quad \text{ as } \quad \beta \to \infty.
\]
\end{remark}

\subsection{Counting on rational subspaces and linear constraints}\label{subsec:sublattice}

Let $L \le \mathbb{Z}^n$ be a primitive\footnote{Equivalently, 
$L = \mathbb{Z}^n \cap E$, where $E := \mathrm{span}(L)$. 
In particular, $L$ has full rank $j := \dim E$ in the Euclidean space $E$, and 
$\det L := \mathrm{vol}_j(E/L)$ denotes the covolume of $L$ with respect to Lebesgue measure on $E$.} 
rank-$j$ sublattice and set $E := \mathrm{span}(L) \in G(j,n)$. 
For $f \in \mathcal{P}_{n,d}$ and $\alpha > 0$, define the sectional counting function
\[
  N_{f,L}(\alpha) ~:=~ \#\{x \in L : f(x) \le \alpha\}.
\]
We will show that the leading term in this lattice–point count is given by the $j$-dimensional volume of the section $[f \le \alpha] \cap E$ divided by $\det L$, with an explicit Lipschitz–type error controlled by intrinsic volumes of the same section.
To avoid confusion, we denote by $ V^{[k]}_j $ the $j$-th intrinsic volume computed within a $k$-dimensional subspace.
This nonstandard double index is introduced for clarity, distinguishing the subspace dimension $k$ from the intrinsic-volume index $j$.

\begin{proposition}[Counting on sublattices]\label{thm:sublattice}
Let $f\in\cP_{n,d}$ and let $L\le\Z^n$ be a primitive rank–$j$ sublattice with $E = \mathrm{span}(L)$. 
Then, for all $\alpha>0$, it holds
\begin{equation}\label{eq:sublattice-main}
  N_{f,L}(\alpha)
  = 
  \frac{\vol_j\big([f\le 1]\cap E\big)}{\det L} \alpha^{j/d}
  + O_j \left(
  \frac{1}{\det L} \sum_{i = 0}^{j - 1} V_{i}^{[j]} \big([f\le 1]\cap E\big) \alpha^{i/d}
  \right).
\end{equation}
Consequently, it holds 
\begin{equation}\label{eq:sublattice-equivalent}
  N_{f,L}(\alpha)
  = \frac{\vol_j\big([f\le 1]\cap E\big)}{\det L}\,\alpha^{j/d}
  + O_{j, f, E}\!\left(\frac{1}{\det L}\alpha^{(j - 1)/d} \right), \quad \text{ as $\alpha \to \infty$}.
\end{equation}
\end{proposition}

\begin{proof}
We shall work inside the Euclidean space $E$ (of dimension $j$), with lattice $L$ whose fundamental domain has volume $\det L$.
Apply Proposition~\ref{prop:disc-V} in dimension $j$ to the convex body
\[
  K_{E, \alpha}
  := [f\le \alpha]\cap E \ =\ [\,\mathscr R_E f\le \alpha\,] = \alpha^{1/d} [ \mathscr R_E f \leq 1]
  \qquad\text{(by \eqref{sec.id})}
\]
to obtain
\[
  \Big|\,\#(K_{E, \alpha} \cap L)-\tfrac{1}{\det L}\vol_j(K_{E, \alpha}) \Big|
  \le C_{j}\!\left(\frac{1}{\det L}\sum_{i=0}^{j-1}V_i^{[j]}(K_{E, \alpha})\right).
\]
Note that the homogeneity of intrinsic volumes gives $V_i^{[j]}(K_{E, \alpha}) = \alpha^{i/d}V_i^{[j]}([f \leq 1] \cap E)$.
Proposition~\ref{thm:sublattice} is proven.
\end{proof}


\begin{corollary}[Linear constraints]\label{cor:linear-constraints}
Let $f \in \cP_{n, d}$, let $A\in\Z^{r\times n}$ have rank $r$.
Set $E:=\ker_\R A$ (so $j=n-r$) and ${L:=\ker A\cap\Z^n}$ (a primitive rank–$j$ sublattice). 
For $b\in\Z^r$, if the set 
\[
  \mathcal S_{A,b}  = \big\{ x\in\Z^n ~:~ Ax=b \big\}
\]
is nonempty, it is a coset $x_0+L$ with $x_0\in\Z^n\cap(A^{-1}b)$ and
\[
  N_{f,A,b}(\alpha)\ :=\ \#\{x\in\mathcal S_{A,b}:\ f(x)\le\alpha\}
  \ =\ \#\big\{y\in L:\ f(x_0+y)\le\alpha\big\}.
\]
Then, for all $\alpha>0$, it holds
\[
 N_{f,A,b}(\alpha)
  = 
  \frac{1}{\det L} \vol_j\big([f \le \alpha] \cap (x_0 + E)\big)
  + O_j 
  \left(
  \frac{1}{\det L}\sum_{i=0}^{j-1}
  V_i^{[j]}\big([f \le \alpha] \cap (x_0 + E)\big)  \right).
\]
Furthermore, in the homogeneous case $b = 0$ (so $x_0 = 0$), the section is linear and it holds
\begin{equation}\label{eq:linear-constraints-asymp}
  N_{f,A,0}(\alpha)
  = \frac{\vol_j\big([f\le 1]\cap E\big)}{\det L}\,\alpha^{j/d}
  + O_{j,f,E}\!\left(
      \frac{1}{\det L}\,\alpha^{(j-1)/d}
    \right), \quad \text{ as $\alpha \to \infty$.}
\end{equation}
\end{corollary}

\begin{proof}
Inside the affine space $x_0+E$ the set $K_{E,\alpha}:=[f\le\alpha]\cap(x_0+E)$ is a $j$–dimensional convex body and $\mathcal S_{A,b}=x_0+L$ is a lattice coset with fundamental domain of volume $\det L$. 
The proof of Proposition~\ref{thm:sublattice} applies verbatim in the affine setting (translate $K_{E,\alpha}$ to $E$), yielding the desired bound.
\end{proof}

\begin{example}[Hyperplane constraint for a quadratic form]\label{ex:hyperplane-quad}\normalfont
Let $Q\in\R^{n\times n}$ be symmetric and positive definite.
Consider $f(x)=x^\top Qx\in\cP_{n,2}$.  
Let $A\in\Z^{1\times n}$ be primitive, meaning that its entries are coprime.
Set
\[
E:=\{x\in\R^n:\ Ax=0\}\in G(n-1,n) \quad \text{ and } \quad L:=\Z^n\cap E.
\]
Thus, $L$ is a primitive rank--$(n - 1)$ sublattice of $E$ with covolume $\det L$.
Applying Proposition~\ref{thm:sublattice} with $j=n-1$, we get, as $\alpha \to \infty$,
\begin{equation}\label{quadra-01}
\#\{x\in L:\ f(x)\le\alpha\}
=\frac{1}{\det L}\,\vol_{n-1}\big([f\le\alpha]\cap E\big)
+O_{n}\!\left(\frac{1}{\det L}\,V_{n-2}^{[n - 1]}\big([f\le\alpha]\cap E\big)\right).
\end{equation}
It follows from Remark~\ref{rem:PiE-quad} and Lemma~\ref{lem.vol_j} that
\begin{equation}\label{quadra-02}
\vol_{n-1}\big([f\le \alpha]\cap E\big)
= \frac{\alpha^{(n - 1)/2}}{\Gamma(1+(n - 1)/2)}\int_E e^{-\mathscr R_E f(y)}\,dy=
\kappa_{n-1} \frac{\alpha^{(n - 1)/2}}{\sqrt{\det(Q|_E)}},
\end{equation}
where $Q|_E$ is the restriction of $Q$ to $E$. 
Furthermore, the homogeneity of $f$ implies
\begin{equation}\label{quadra-03}
    V_{n-2}^{[n - 1]} \big([f\le \alpha]\cap E\big)
    = \alpha^{\frac{n-2}{2}}V_{n-2}^{[n - 1]}\big([f\le 1]\cap E\big).
\end{equation}
Combining \eqref{quadra-01}, \eqref{quadra-02} and \eqref{quadra-03}, we finally obtain
\[
\#\{x\in\Z^n:\ Ax=0,\ x^\top Qx\le \alpha\}
= \frac{\kappa_{n-1} \alpha^{(n - 1)/2}}{\det L\sqrt{\det(Q|_E)}}
+ O_{n, Q, A}
\Big(\frac{\alpha^{(n - 2)/2}}{\det L}\Big), \,\, \text{ as } \alpha \to \infty.
\]
\end{example}

\subsection{Theta function asymptotics}

Define the Epstein–type theta series associated with a function $f$ by
\[
  \Theta_f(t) := \sum_{x \in \mathbb{Z}^n} e^{-t\,f(x)}, \quad \text{ for every } t > 0.
\]
Heuristically, the leading behaviour of $\Theta_f(t)$ as $t \searrow 0$
is governed by the volume of the level set of $f$:
a Poisson summation argument (when applicable) or the standard volume heuristic suggests that, 
\[
  \Theta_f(t)
  \sim t^{-n/d} 
  \int_{\mathbb{R}^n} e^{-f(x)}\,dx \quad \text{ as $t \searrow 0$},
\]
as discussed in the general references~\cite{Kratzel1989,Huxley1996}.
The following proposition makes this asymptotic explicit.

\begin{proposition}[Small-scale asymptotics]\label{prop:theta}
Let $f\in\cP_{n,d}$. 
Then, it holds
\[
  \Theta_f(t) =
  \Gamma(1 + n/d) \opV_n(f)\,t^{-n/d}
  +\ O_n\!\left(\Big(\sum_{j=0}^{n-1}\opV_j(f)\Big)\,t^{-(n-1)/d}\right) \quad \text{as $t \searrow 0$}.
\]
Equivalently, it holds
\[
\Theta_f(t)=\Gamma\!\Big(1+\frac{n}{d}\Big)\,\opV_n(f)\,t^{-n/d}
\ +\ O_n\!\big(\mathds W(f)\,t^{-(n-1)/d}\big) \quad \text{as $t \searrow 0$},
\]
where $\mathds W(f):=\max_{0\le j\le n-1}\opV_j(f)$.
\end{proposition}

\begin{proof}
Observe first that 
\[
    N_f(s) = \# \big\{ x \in \Z^n ~:~ f(x) \leq \alpha \big\}
\]
is the pushforward of the counting measure  on $\mathbb{Z}^n$ by $f$.
It follows that, as a Lebesgue–Stieltjes integral,
\[
  \Theta_f(t)
  = \sum_{x\in\mathbb{Z}^n} e^{-t f(x)}
  = \int_0^\infty e^{-t s}\,dN_f(s).
\]
Notice that $N_f( 0) = 1$ (only the origin has value $0$) and $\textstyle\lim_{s \to \infty} e^{-ts}N_f(s) = 0$ since $N_f(s) \ll 1 + s^{n/d}$. 
Therefore, applying the Stieltjes integration by parts to the case $e^{-ts}$ and $N_f(s)$, we obtain
\begin{equation}\label{THETA}
    \begin{split}
    \Theta_f(t)
  &= \big[e^{-ts}N_f(s)\big]_{0}^{\infty}
  + t\int_0^\infty e^{-ts}N_f(s)ds\\
  &= -N_f(0) + t\int_0^\infty e^{-ts}N_f(s)ds\\
  &= t\int_0^\infty e^{-ts}N_f^\ast(s)ds.
  \end{split}
\end{equation}
with $N_f^\ast=N_f-1$. Applying Proposition~\ref{thm:lattice-asymp}, we have
\[
  N_f^\ast(s)
  = \opV_n(f) s^{n/d}
    + O_n \!\Big(\sum_{j=0}^{n-1} \opV_j (f) s^{j/d}\Big), \quad \text{ for every } s > 0.
\]
Inserting this into the identity~\eqref{THETA} and using the formula
\[
  \int_0^\infty e^{-ts}s^\beta\,ds \;=\; \Gamma(\beta+1)\,t^{-\beta-1}
  \quad \text{ for } \beta > - 1,
\]
we arrive at
\[
\begin{aligned}
  \Theta_f(t)
    = & ~ t\opV_n(f) \int_0^\infty e^{-ts}s^{n/d}\,ds
     + t  \cdot O_n \Bigg(
     \sum_{j=0}^{n-1}\opV_j(f)\int_0^\infty e^{-ts}s^{j/d}\,ds
     \Bigg)
     \\
    = & ~ 
    \Gamma \big(1+ n/d \big) \opV_n(f)\,t^{-n/d}
     +
     O_n
     \Bigg(
     \sum_{j=0}^{n-1}\Gamma\big(1+j/d\big) \opV_j(f)\,t^{-j/d}
     \Bigg).
\end{aligned}
\]
As $t \searrow 0$, the dominant error term corresponds to $j = n-1$, whence
\[
  \Theta_f(t)
  = 
  \Gamma\big(1+ n/d \big) \opV_n(f)t^{-n/d}
    + O_n\!\big(\mathds W(f) t^{-(n-1)/d}\big),
\]
which completes the proof.
\end{proof}

Our theta asymptotic records only the first term with a boundary‑driven remainder.  
For quadratic $f$ one has modular/Poisson structures leading to finer expansions; for a general homogeneous $f$ we do not attempt second‐order terms.  
To end, let us conclude with the classical Gauss circle problem.

\begin{example}[Gauss circle problem, that is, $n=2$, $d=2$ and $f(x, y) = x^2+y^2$]\label{ex:circle}\normalfont
In this case, $[f\le\alpha]$ is the disk of radius ${R=\sqrt\alpha}$.
Then, one directly has
\[
\opV_2(f)=\frac{1}{\Gamma(2)}\int_{\R^2} e^{-(x^2+y^2)}\,dx\,dy=\pi.
\]
It follows from Proposition~\ref{thm:lattice-asymp} that
\[
N_f(\alpha)=\pi\,\alpha\ +\ O\!\big(\sqrt\alpha\big), \quad \text{ as $\alpha \to \infty$}
\]
and from Proposition~\ref{thm:primitive} on counting primitive points, that
\[
\#\{(x,y)\in\Z^2_{\rm prim}:\ x^2+y^2\le\alpha\}
=\frac{\pi}{\zeta(2)}\,\alpha\ +\ O \!\big(\sqrt\alpha\log\alpha\big) \quad \text{ as $\alpha \to \infty$}.
\]
Concerning the expansion for the theta series, Proposition~\ref{prop:theta} yields
\[
\Theta_f(t)\ =\ \frac{\pi}{t}\ +\ O\!\big(t^{-1/2}\big) \quad \text{ as } t \searrow 0.
\]
\end{example}

\appendix

\section{Appendix: On the choice of $\cP_{n, d}$}
We choose the positive cone $\cP_{n,d}$ as our main setting because the theory of intrinsic volumes has been developed for convex bodies.
Remark that the notion of dual volumes is also meaningful even for star bodies. 
Working within $\cP_{n,d}$ guarantees that both intrinsic and dual volumes are well defined, since the sublevel sets of any $f \in \cP_{n,d}$ are convex bodies. 
In this sense, $\cP_{n,d}$ provides a natural unification: it offers a single class in which both intrinsic and dual volumes can be consistently defined.
The following propositions make this point precise.

\begin{proposition}\label{prop.cF-coer}
Let $f \not\equiv 0$ be a lower semicontinuous nonnegative and positively $d$--homogeneous function.
Then, the following assertions are equivalent:
\begin{itemize}
	\item[(i)] $[f \leq 1]$ is bounded (equivalently, compact);
	\item[(ii)] $\textstyle\min_{e \in \Sph^{n - 1}} f(e) > 0$;
	\item[(iii)] there exists $\varpi > 0$ such that $f(x) \geq \varpi \| x \|^d$, for every $x \in \R^n$;
	\item[(iv)] (definition of $\cP_{n, d}$) $f(x) > 0$ for every $x \neq 0$.
\end{itemize}
\end{proposition}

\begin{proof}
\noindent\textit{(i) $\Longrightarrow$ (ii).}
Assume $[f\le1]$ is bounded. Consider the restriction of $f$ to the compact set $\Sph^{n-1}$.
Since $f$ is lower semicontinuous and $\Sph^{n-1}$ is compact, there exists $e_0 \in \Sph^{n - 1}$ such that $\varpi := \textstyle\min_{\|e\|=1} f(e) = f(e_0) \geq 0$.
If $\varpi = 0$, then $f(e_0)  = 0$ and by homogeneity
$f(te_0)=t^d f(e_0)=0$ for all $t > 0$.
Hence, we have the inclusion $\{t e_0:\ t>0\}\subset [f\le1]$.
This contradicts the boundedness of $[f\le1]$. 
Therefore $\varpi > 0$.

\medskip
\noindent\textit{(ii) $\Longrightarrow$ (iii).}
Choose $\varpi = \textstyle\min_{\|e\|=1} f(e)>0$. 
The case $x = 0$ is vacuous.
For any $x \neq 0$, by homogeneity, we have
\[
f(x)=\|x\|^d f(x/\|x\|)\ \ge\ \varpi \,\|x\|^d.
\]

\medskip
\noindent\textit{(iii) $\Longrightarrow$ (iv).}
If \textit{(iii)} holds true, then for $x\neq0$ we clearly have $f(x)\ge \varpi \|x\|^d>0$.

\medskip
\noindent\textit{(iv) $\Longrightarrow$ (i).}
By the lower semicontinuity of $f$ and the compactness of $\mathbb S^{n - 1}$, we have $\varpi := \textstyle\min_{\|e\|=1} f(e) = f(e_0)$ for some $e_0 \in \Sph^{n - 1}$.
Then, \textit{(iv)} implies that $\varpi = f(e_0) > 0$.
As we have shown in  the implication \textit{(ii)$\Longrightarrow$(iii)}, for all $x$ we have
$f(x)\ge \varpi \,\|x\|^d$. 
Therefore, we obtain
\[
[f\le1]\ \subset\ \big\{ x: \varpi \,\|x\|^d\le 1 \big\}
= \big\{x:\ \|x\|\le \varpi^{-1/d} \big\},
\]
which is bounded. 
This completes the proof.
\end{proof}

\begin{proposition}
    Let $f: \R^n \to [0, + \infty)$ be a convex and positively $d$--homogeneous function.
    Then, the following assertions are equivalent:
    \begin{itemize}
        \item[(i)] $f$ is convex;

        \item[(ii)] $[f \leq 1]$ is convex.
    \end{itemize}
\end{proposition}

\begin{proof}
The implication $(i) \Longrightarrow (ii)$ is direct and so, it suffices to check $(ii) \Longrightarrow (i)$.
Suppose that $K = [f \leq 1]$ is convex.
Recall that the gauge function of $K$ is defined by
\[
    \rho_K(x) = \inf \big\{ t > 0: x \in t K  \big\} \in [0, + \infty)
\]
is convex and positive $1$--homogeneous.
Since $f$ is $d$--homogeneous, we have $x \in tK$ if and only if $f(x) \leq t^d$. 
Therefore, we have $\rho_K(x) = f(x)^{1/d}$ and equivalently $f(x) = \rho_K(x)^d$.
Thanks to the convexity of $\rho_K$ and using the fact that $s \mapsto s^d$ is convex and nondecreasing on $[0, + \infty)$, we infer that $f$ is convex. 
This completes the proof.
\end{proof}

\paragraph{Acknowledgments.}
Khai-Hoan Nguyen-Dang thanks Phu Nhan Chung for his interest in this work and for the invitation to speak.
He also gratefully acknowledges the support of the Morningside Center of Mathematics, Chinese Academy of Sciences.
Tr\'i Minh L\^e is partially supported by the FWF (Austrian Science Fund) grant \texttt{DOI 10.55776/P-36344N}.
We thank Manh Hung Le and Phan Quoc Bao Tran for their early interest and helpful discussions at the beginning of this project.

\nocite{*}
\bibliographystyle{siam}
\bibliography{biblio}


\noindent Tr\'i Minh L\^E

\medskip

\noindent Institut f\"{u}r Stochastik und Wirtschaftsmathematik, VADOR E105-04
\newline TU Wien, Wiedner Hauptstra{\ss }e 8, A-1040 Wien\medskip
\newline\noindent E-mail: \texttt{minh.le@tuwien.ac.at}
\newline\noindent\texttt{https://sites.google.com/view/tri-minh-le} \smallskip\newline
\noindent Research supported by the FWF (Austrian Science Fund) grant \texttt{DOI 10.55776/P-36344N}.

\bigskip

\noindent Khai--Hoan NGUYEN--DANG

\medskip

\noindent Morningside Center of Mathematics, Chinese Academy of Sciences, Beijing, China
\newline No. 55, Zhongguancun East Road, Haidian District, Beijing 100190 \medskip
\newline\noindent E-mail: \texttt{khaihoann@gmail.com}
\newline\noindent\texttt{https://sites.google.com/view/nguyen-dang-khai-hoan/home} \smallskip\newline
\noindent
\end{document}